\documentclass[11pt]{article}
\usepackage[utf8]{inputenc}

\usepackage{amsmath} 
\usepackage{amsthm,amssymb,mathrsfs} 
\usepackage[T1]{fontenc}
\usepackage[utf8]{inputenc}
\usepackage{ae,aecompl}
\usepackage{times}
\usepackage{cite}

\usepackage[colorlinks=true,
            linkcolor=blue,
            urlcolor=cyan,
            citecolor=green]{hyperref}

\usepackage{verbatim}
\usepackage{pgfplots}
\pgfplotsset{compat=1.7}
\usepackage{tikz}
\usepackage[numeric,backrefs]{amsrefs}
\usepackage{enumerate}
\usepackage[cmtip, all]{xy}
\usepackage{booktabs}
\usepackage{pb-diagram}
\usepackage{bbm}
\usepackage[top=2.54cm,bottom=2.54cm,left=2.54cm,right=2.54cm]{geometry}
\usepackage{tikz-cd} 
\numberwithin{equation}{section} 
\usepackage{bbding}
\usepackage{multicol, color}
 \usepackage{array, multirow} 
 \usepackage{arydshln}
 \usepackage{slashed}

\newcommand{\rF}{{\rm F}}

\newcommand{\rI}{{\rm I}}

\newcommand{\rL}{{\rm L}}

\newcommand{\bN}{{\bf N}}

\newcommand{\cA}{\mathcal{A}}

\newcommand{\cE}{\mathcal{E}}

\newcommand{\cM}{\mathcal{M}}

\newcommand{\cS}{\mathcal{S}}

\newcommand{\sX}{\mathscr{X}}

\newcommand{\fg}{{\mathfrak g}}

\newcommand{\fh}{{\mathfrak h}}

\newcommand{\fk}{{\mathfrak k}}

\newcommand{\fm}{{\mathfrak m}}

\newcommand{\fp}{{\mathfrak p}}

\newcommand{\fs}{{\mathfrak s}}

\newcommand{\fu}{{\mathfrak u}}

\newcommand{\N}{\bN}

\newcommand{\R}{\mathbb{R}}
\newcommand{\C}{\mathbb{C}}

\newcommand{\su}{\mathfrak{su}}
\newcommand{\so}{\mathfrak{so}}

\newcommand{\gl}{\mathfrak{gl}}
\newcommand{\fsl}{\mathfrak{sl}}

\newcommand{\SO}{{\rm SO}}

\newcommand{\SU}{{\rm SU}}

\newcommand{\U}{{\rm U}}

\newcommand{\Ad}{\mathrm{Ad}}

\renewcommand{\det}{\mathop\mathrm{det}\nolimits}

\newcommand{\End}{{\mathrm{End}}}

\renewcommand{\epsilon}{\varepsilon}

\newcommand{\ad}{\mathrm{ad}}

\newcommand{\diag}{\mathrm{diag}}

\renewcommand{\Im}{\mathop{\mathrm{Im}}}

\renewcommand{\Re}{\mathop{\mathrm{Re}}}

\newcommand{\tr}{\mathop{\mathrm{Tr}}\nolimits}

\newcommand{\wt}[1]{\widetilde{#1}}

\newcommand{\sym}{sym}

\newcommand{\fraksp}{\mathfrak{sp}}

\newcommand{\cym}{\mathcal{YM}}

\newcommand{\qandq}{\quad\text{and}\quad}
\newcommand{\qforq}{\quad\text{for}\quad}
\newcommand{\qwithq}{\quad\text{with}\quad}
\newcommand{\qwhereq}{\quad\text{where}\quad}

\def\<{\mathopen{}\left<}
\def\>{\right>\mathclose{}}
\def\({\mathopen{}\left(}
\def\){\right)\mathclose{}}

\definecolor{gold}{rgb}{0.85,.66,0}
\definecolor{cherry}{rgb}{0.9,.1,.2}
\definecolor{burgundy}{rgb}{0.8,.2,.2}
\definecolor{orangered}{rgb}{0.85,.3,0}
\definecolor{orange}{rgb}{0.85,.4,0}
\definecolor{olive}{rgb}{.45,.4,0}
\definecolor{lime}{rgb}{.6,.9,0}
\definecolor{green}{rgb}{.2,.7,0}
\definecolor{grey}{rgb}{.4,.4,.2}
\definecolor{brown}{rgb}{.4,.3,.1}

\newtheorem{theorem}{Theorem}[section]
\theoremstyle{Theorem}

\newtheorem{proposition}{Proposition}[section] 
\theoremstyle{Proposition}

\newtheorem{corollary}[proposition]{Corollary}  
\newtheorem{lemma}[proposition]{Lemma}  

\theoremstyle{remark} 
\newtheorem{remark}[proposition]{Remark} 

\theoremstyle{definition}
\newtheorem{definition}[proposition]{Definition}

\DeclareMathOperator{\Gl}{Gl}
\DeclareMathOperator{\Tr}{Tr}

\DeclareMathOperator{\Span}{ {\rm Span}}

\newcommand{\ig}{{\textsl g}}

\usetikzlibrary{shapes.geometric}

\newcommand{\gt}{\mathrm{G}_2}
\newcommand{\V}{V^{5,2}}
 \newcommand{\Spi}{
  \slashed{S}
}
\newcommand{\Hpi}{
  \slashed{H}
}
\usepackage[textsize=footnotesize,textwidth=3.5cm]{todonotes}
\usepackage[affil-it]{authblk}

\AtBeginDocument{%
   \def\MR#1{}
}

\begin{document}

\title{Homogeneous $\gt$ and Sasakian instantons on the Stiefel $7$-manifold}

\author[1]{Andrés J. Moreno}
\author[2]{Luis E. Portilla}

\affil[1]{Universidade Estadual de Campinas (Unicamp), SP, Brazil}

\affil[2]{\textit{Current affiliation:} University of Cauca, Popayán, Colombia
\\ 

\textit{Former affiliation:} Université de Brest, CNRS UMR 6205, LMBA, F-29238 Brest, France}
\date{\today}

\maketitle
\begin{abstract} 
We study homogeneous instantons on the seven--dimensional Stiefel manifold
$\V$ in the context of $\gt$ and Sasakian geometry. According to the
reductive decomposition of $\V$, we provide an explicit description of all
invariant $\gt$ and Sasakian structures. In particular, we characterise the
invariant $\gt$--structures inducing a Sasakian metric, among which the
well--known nearly parallel $\gt$--structure (Sasaki--Einstein) is included.
As a consequence, we classify the invariant connections on homogeneous
principal bundles over $\V$ with gauge group $\U(1)$ and $\SO(3)$, satisfying
either the $\gt$ or the Sasakian instanton condition. In addition, we study infinitesimal deformations of $\gt$--instantons on
coclosed $\gt$--manifolds using a spinorial approach. By means of a
Weitzenb\"ock--type formula with torsion, we obtain curvature obstructions
to the existence of non--trivial infinitesimal deformations and prove
rigidity results for certain homogeneous $\gt$--instantons.
\end{abstract}
\section{Introduction}\label{sec:intro}
The Stiefel manifold $V^{n,k}$ of orthonormal $k$-frames in the $n$-dimensional Euclidean space (i.e. $1\leq k\leq n$) admits a transitive $\SO(n)$-action with isotropy subgroup isomorphic to $\SO(n-k)$, thus, it can be written homogeneously as  $V^{n,k}=\SO(n)/\SO(n-k)$ and consequently, its family of $\SO(n)$-invariant Riemannian metrics is described in terms of the isotropy representation of $V^{n,k}$. For instance, for $k=1$ or $k=n$, the isotropy representation is irreducible (indeed $V^{n,1}=S^{n}$ or $V^{n,n}=\SO(n)$), in those cases, there is a unique $\SO(n)$-invariant metric, which is Einstein \cite{wolf1968}. 
For $k=2$, the isotropy representation is reducible and the space of $\SO(n)$-invariant Riemannian metrics is described by a $4$-dimensional family, within which there exists a unique Einstein metric, which is inherited from the Grassmannian  $G_2(\R^n)$ (e.g. \cite{Arvanitoyeorgos1996,kerr1998new}). 
In addition, $V^{n,2}$ fits as an instance of the Boothby-Wang construction, as the total space of a $S^1$-principal bundle   \cite{boothby1958}*{Theorem 3}, moreover, it has a regular contact structure whose associated Reeb vector field generates a one-parameter subgroup of right translations of $V^{n,2}$.

A \emph{contact structure}\footnote{The reader unfamiliar with Sasakian geometry can find a complete study on this subject   in \cite{Boyer2008}} on a  $(2n+1)$-dimensional manifold $M$ is a 1-form   satisfying $\eta\wedge (d\eta)^n\neq 0$, in this case there exists a unique vector $\xi\in \sX(M)$, called   \emph{Reeb vector field},   such that  $ \eta(\xi)=1$.  The contact structure is called an \emph{almost contact structure} if there exists an endomorphism  $\Phi\in \End(TM)$ which satisfy 
$\Phi^2=-\rI_{TM}+\eta\otimes \xi.$
A  Riemannian metric $g$ is called compatible with the contact structure, if $g$ satisfies 
$g(\Phi X,\Phi Y )= g(X,Y)-\eta(X)\eta(Y),$
in this case, the contact structure is called an \emph{almost contact metric structure}. In addition, if $\cS=(\eta,\xi,g,\Phi)$ satisfies $2g(X,\Phi(Y))=d\eta(X,Y)$, the quadruple $\cS$ is called a \emph{contact metric structure}.  Furthermore, if the Nijenhuis tensor $N_\Phi$ induced by $\Phi$ satisfies $N_\Phi=-d\eta\otimes \xi$, then $\cS$ is called a \emph{Sasakian structure}.  In this paper, we consider the $7$--dimensional  Stiefel manifold $\V$  endowed with a one-parameter family of $\SO(5)$-invariant Sasakian structures $\cS=(\eta,\xi,g,\Phi)$. For instance, homogeneous contact/Sasakian structures have been studied in \cite{andrada2009,correa2019} and \cite{perrone1998homogeneous}.\\

On the other hand, a differential $3$-form $\varphi$ on an oriented $7$--dimensional manifold $M$ is called a $\gt$--structure
if at each $p\in M$, there is a basis $e_1,\dots, e_7$ of $T_pM$ such that
\begin{align}\label{eq: def G2-structure}
\varphi_p=e^{125}+e^{136}+e^{147}+ e^{234}-e^{267}+e^{357}-e^{456},
\end{align}
where $e^1, \dots , e^7$  denotes the dual basis  of $T_p^\ast M$ and $e^{ijk}=e^i\wedge e^j\wedge e^k$. In particular, $\varphi$ determines a Riemannian metric $g_\varphi$, we study the family of $\SO(5)$-invariant $\gt$-structures on $\V$, such that the induced  metric $g_\varphi$ is Sasakian. With respect to these structures, we study the theoretical gauge instanton equations associated with the underlying Sasakian and $\gt$-geometry of $\V$. Namely, on a $G$--principal bundle  $P\to M^7$ with $\sigma\in\Omega^{3}(M)$, the $1$-form connection $A\in \mathcal{A}(P)$ is called an $\lambda$-\emph{instanton}, if it satisfies   (e.g. \cites{donaldson1998gauge,Tian2000}) 
\begin{equation}
\label{eq:introInstEquation}
\ast(\sigma\wedge F_A)=\lambda F_A , \quad\text{where} \quad\lambda\in\R 
\end{equation} 
where $F_A=dA+A\wedge A$ denotes the corresponding curvature $2$-form. According to the specific pair $(\sigma,\lambda)$ in \eqref{eq:introInstEquation}, we define:
\begin{enumerate}
\item 
A \textit{Self-dual  contact instantons (SDCI) \cite{portilla2023}} is a solution  of   \eqref{eq:introInstEquation}  with  $\lambda=1$ and  $\sigma:=\eta\wedge \omega$ naturally  defined by a contact structure of $(M,\eta)$, where $\omega:=\frac{1}{2}d\eta$ is called the \emph{fundamental $2$-form}.
\item 
A \textit{$\gt$--instantons \cite{corrigan1983first,donaldson1998gauge}} is a  solution of  \eqref{eq:introInstEquation} with $\lambda=-1$ and $\sigma=\varphi$ a \emph{coclosed} (i.e. $d\ast\varphi=0$) $\gt$--structure on $M$. Equivalently, a connection $A$ is a $\fg$--instanton if and only if
\begin{equation}\label{eq:G2instantonIntro}
    F_A\wedge \ast\varphi=0
\end{equation}
\end{enumerate}    
Furthermore, for  $E\to (M^7,\cS)$ a \textit{Sasakian  holomorphic  bundle} \cite{biswas2010vector},  the usual notion of Hermitian Yang-Mills (HYM) connection on a  Kähler manifolds  is extended to odd dimension in the natural fashion to the so-called \emph{transverse Hermitian Yang-Mills connections (tHYM)} (see  \eqref{eq:tHYM}). The concept of Chern connections also extends naturally to a Sasakian version, as a connection mutually compatible with the transverse holomorphic structure and the Hermitian metric. A noteworthy result is that if $E\to M$ is a holomorphic Sasakian bundle on a contact Calabi-Yau manifold $M$, along compatible connections the notion of SDCI, $\gt$-instantons and tHYM coincide. cf. \cite[Theorem~1.1]{portilla2023} and \cite[Lemma ~21]{calvo2020gauge}.  

In \cites{donaldson1998gauge,Donaldson2011},  Donaldson, Segal and Thomas  suggested  that it may be possible to use  $\gt$-instantons  to construct an enumerative invariant of $\gt$-manifolds. For instance, when $(M^7,\varphi)$ has a torsion-free $\gt$-structure $\varphi$ (i.e. $d\varphi=0$ and $d\ast\varphi=0$ \cite{fernandez1982riemannian}), solutions to the $\gt$-instanton equation have been studied  in \cite{sa2015g2,walpuski2015g_2,clarke2014instantons,lotay20182,singhal2022}. However, the construction of metrics with holonomy $\gt$ has proven to be a challenging task, thus, the torsion-free condition has been relaxed to the class of coclosed $\gt$-structures \cite{ball2019gauge,clarke2023}, which has a notable application   in theoretical physics to the Killing spinor equations in supergravity (cf.  \cite{Ivanov2005,Ivanov10,delaOssaG21,Lotay2022,Fernandez2011}). Analogously, Wang has proposed studying the compactification of the moduli space of contact instantons as a tool to construct a Donaldson-type invariant for contact $7$-manifolds 
\cite{wang2012higher}.\\

\textbf{Description of the main results: }
Let $\so(5)= \so(3)\oplus \fm$ be a reductive decomposition of $\V$, in \cite{reidegeld2010spaces} Reidegeld proved that the space of $\SO(5)$-invariant $\gt$-structures has dimension $5$ and the space of induced $\SO(5)$-invariant Riemannian metrics has dimension $4$. Thus, we prove in Lemma \ref{lem: SO(5) invariant G2 structures}  that the $\Ad(\SO(3))$-invariant $3$-form \eqref{eq: on G2-form} determined by  $(a,b,x,y,z)\in \R^5$ defines a $\gt$-structure (up to $\gt$-transformation), if 
$$
a^2+by>0, \quad b^2+ax>0 \quad z\neq 0 \qandq xy=ab.
$$
In particular, it defines a family of (non-purely) coclosed $\gt$-structures, within the nearly parallel $\gt$-structures arises as an $S^1$--family for $x=a$, $y=b$ and $a^2+b^2=z^2=\frac{27}{512}$ (cf. Theorem \ref{thm:Nearly}).

In addition, for any $y_1,y_2,y_3\in \R^+$, Theorem \ref{thm:V52 Sasakian1} characterises the $\SO(5)$-invariant  Sasakian metrics as 
$$
y_1=4y_2^2 \qandq y_2=y_3, 
$$ 
as a consequence, for a given $(z,a,b,x,y)\in \R^5$, we describe the family of $\SO(5)$-invariant $\gt$-structures inducing a Sasakian metric (cf. Corollary \ref{cor:V52 Sasakian}).

In Section \ref{sec:InstantonCOnstruction}, we revisit the interaction between SDCI, $\gt$-instantons and tHYM connections addressed in \cite[Proposition~3.11]{portilla2023} and  \cite[Lemma~21]{calvo2020gauge} against the backdrop of \textit{contact Calabi-Yau} manifolds. For a bundle $E\to \V$, in Lemma  \ref{lem:sdci g2} we prove that if the connection $A$ is SDCI then it is a $\gt$-instanton. Conversely, for a holomorphic Sasakian bundle $E\to \V$, if the Chern connection is a $\gt$-instanton, then it is a SCDI. Finally, we prove that HYM connections on a holomorphic bundle  $E\to G_{2}(\R^5)$ on the Grassmannian, pulling back to $\gt$-instantons on $\pi^\ast E\to \V$ (cf. Corollary \ref{cor:tHYM}). We point out that the key argument can be carried out under the sole assumption of an $\mathrm{SU}(3)$-structure, thereby improving the results mentioned above.

In Section \ref{sec:homogeneousBundles}, we provide examples of homogeneous SDCI and $\gt$-instantons over $\V$, with gauge group $\SO(3)$ and $\U(1)$. In particular, we prove that the $\SO(3)$-invariant $\gt$-instanton \eqref{eq: invariant_connection} is a Yang-Mills connection on $(\V,\varphi)$, for any $\SO(5)$-invariant coclosed $\gt$-structure (cf. Theorem \ref{thm: Yang-Mills example}).

Finally, in Section~\ref{sec:spinors} we extend the spinorial framework proposed by R. Singhal \cite{singhal2022} to coclosed $\gt$–structures,    and as a concrete application in Corollary \ref{prop:rigidity}, we prove rigidity for a family of homogeneous $\gt$–instantons on the Stiefel manifold $V^{5,2}$ for a non-empty range of the metric parameter $y_2$.\\

\textbf{Outline:} In Section \ref{homogeneous space}, we provide some preliminaries about $\V$ as a homogeneous space, with the aim of characterising all the $\SO(5)$-invariant coclosed $\gt$-structures and the associated Sasakian metrics. In Section \ref{sec:InstantonCOnstruction}, we explore the interaction between SDCI and $\gt$--instantons on $\V$, in particular, Lemmata \ref{lem:sdci g2}, \ref{lem:g2 sdci} and Corollary \ref{cor:tHYM} indicate conditions such that contact instantons and  $\gt$-instantons coincide.  Finally, in Section \ref{sec:homogeneousBundles} we consider homogeneous principal bundles on $\V$ and we construct examples of invariant instantons on $\SO(3)$, $U(1)$--homogeneous bundles. In particular, in Section \ref{sec:SO(3)Bundles} we characterise all invariant $\SO(3)$--connection and we show that these satisfy the $\gt$ and SDCI condition for appropriate parameters. Finally, in Section \ref{sec:spinors} we study the deformation theory for a particular example of a  $\gt$-instanton, specifically in Corollary \ref{prop:rigidity} we prove that it is rigid for a non-empty interval of the metric parameter $y_2$.\\

\textbf{Acknowledgments and funding:} AM was funded by the São Paulo Research Foundation [2021/08026-5] and [2023/13780-6] (Fapesp). LP was funded  by Centre Henri Lebesgue (CHL) under the grant \textbf{ANR-11-LABX-0020-01}. He also thanks the LMBA at Université de Bretagne Occidentale (UBO) for its hospitality
\section{Homogeneous $\gt$ and contact metric structures on $\V$}
\label{homogeneous space}
In this section, we provide an explicit description of a $\Ad(\SO(3))$-invariant $\gt$-structure and the induced inner product on $\V$. We also establish the notation that we will use throughout  this work.
\subsection{ $\gt$-structures on $\V$} 
In order to describe the space of $\SO(5)$-invariant $\gt$-structures on  $\V=\SO(5)/\SO(3)$\footnote{This subject is discussed in depth in \cite{friedrich1997nearly}.}, we fix a reductive decomposition of $\V$ as a homogeneous space.   Consider the Lie algebra of $\SO(5)$
\begin{equation}\label{eq: so(5)_Lie_algebra}
    \so(5):=\{A\in \gl(\R^5) \; | \; A+A^T=0\},
\end{equation}

Let $E_{ij}$ denote the elementary $5\times5$ matrix with a $1$ in the $(i,j)$-entry and zeros elsewhere.
We fix the following basis of the Lie algebra $\so(5)$:

\begin{align}\label{eq: so(5)_basis}
\begin{split}
e_1=&E_{12}-E_{21},\quad e_4=E_{15}-E_{51},  \quad  e_7=E_{25}-E_{52},\quad e_{10}=E_{54}-E_{45}, \\[0.3em]
e_2=&E_{13}-E_{31},\quad  e_5=E_{23}-E_{32}, \quad  e_8=E_{34}-E_{43}, \quad \\[0.3em]
e_3=&E_{14}-E_{41},\quad e_6=E_{24}-E_{42},  \quad  e_9=E_{35}-E_{53}.
\end{split}
\end{align} 
We consider the Lie algebra  $\so(3)$ embedded in $\so(5)$ which is spanned by
$\{e_8,e_9,e_{10}\}$. The Lie bracket of $\so(5)$ is given by the commutator described in Table \ref{tab:bracket}. 
\begin{table}[ht]
\begin{center}
\newcolumntype{M}{>{$}c<{$}}
\begin{tabular}{|M||M|M|M|M|M|M|M|M|M|M|}
 \hline
 [\cdot,\cdot]  & e_1 & e_2 & e_3   & e_4   & e_5   & e_6   & e_7   & e_8   & e_9   & e_{10}\\  \hline \hline     
           e_1  & 0   &-e_5 &-e_6   &-e_7   & e_2   & e_3   & e_4   & 0     & 0     & 0     \\  \hline
           e_2  & e_5 & 0   &-e_8   &-e_9   &-e_1   & 0     & 0     & e_3   & e_4   & 0     \\  \hline
           e_3  & e_6 & e_8 & 0     &-e_{10}& 0     &-e_1   & 0     &-e_2   & 0     & e_4   \\  \hline
           e_4  & e_7 & e_9 & e_{10}& 0     & 0     & 0     &-e_{1} & 0     &-e_2   &-e_3   \\  \hline
           e_5  &-e_2 & e_1 & 0     & 0     & 0     &-e_8   &-e_9   & e_{6} & e_7   & 0     \\  \hline
           e_6  &-e_3 &  0  & e_1   & 0     & e_8   & 0     &-e_{10}&-e_5   & 0     & e_7   \\  \hline
           e_7  &-e_4 &  0  & 0     & e_1   & e_9   & e_{10}& 0     & 0     &-e_5   &-e_6   \\  \hline
           e_8  & 0   &-e_3 & e_2   & 0     &-e_6   & e_5   & 0     & 0     &-e_{10}& e_9   \\  \hline
           e_9  & 0   &-e_4 & 0     & e_2   &-e_7   & 0     & e_5   & e_{10}& 0     &-e_8   \\  \hline
        e_{10}  & 0   & 0   &-e_4   & e_3   & 0     &-e_7   & e_6   &-e_9   & e_8   & 0     \\
\hline
\end{tabular}
\caption{\label{tab:bracket}Bracket of $\so(5)$ with respect to the basis \eqref{eq: so(5)_basis}}
\end{center}
\end{table}

On $\so(5)$, we fix the inner product $\langle a,b\rangle=Q(a,b)=\frac{1}{2}\tr(ab^T),$ 
with respect to which, we have  the orthogonal splitting 
\begin{equation}\label{eq: reductive decomposition so(5)}
\so(5)= \fm\oplus\so(3)\qwhereq \fm:=\so(3)^\perp=\Span\{e_1,...,e_{7}\}.
\end{equation}
Note that in terms of the basis \eqref{eq: so(5)_basis}, the restriction of $\ad(e_k)_{\fm}$ to $\fm$ are given by   
\begin{align*}
\ad(e_8)_\fm    &=\diag(0,A_1,A_1),\\
\ad(e_9)_\fm    &=\diag(0,A_2,A_2),\\
\ad(e_{10})_\fm &=\diag(0,A_3,A_3),
\end{align*}  
where 
{\footnotesize $
A_1=\begin{pmatrix}
  0 &-1 & 0 \\
  1& 0 & 0  \\
  0 & 0 & 0    
\end{pmatrix}
$, \quad  $
A_2=\begin{pmatrix}
  0 &0&-1 \\
  0& 0 & 0  \\
  1& 0 & 0    
\end{pmatrix}
$    and $
A_3=\begin{pmatrix}
  0& 0 & 0  \\
  0& 0 &-1  \\
  0& 1 & 0   
\end{pmatrix}
$ }
are the canonical basis of $\so(3)$. 
Denote by 
$ 
\fm_1=\Span\{e_{1}\},$  $\fm_2:=\Span\{e_2,e_3,e_4\}$ and $ \fm_3:=\Span\{e_5,e_6,e_7\}.
$
By using the Lie bracket in Table \ref{tab:bracket}, it is straightforward to verify  that
\begin{equation*}
    \ad(\so(3)) \fm_1=0, \quad \ad(\so(3)) \fm_2\subseteq \fm_2  \qandq
    \ad(\so(3)) \fm_3\subseteq \fm_3.  
\end{equation*}
Consequently, for the reductive decomposition \eqref{eq: reductive decomposition so(5)} we obtain the irreducible $\ad(\so(3))$-submodules 
\begin{equation}\label{eq: irreducible so(3)-decomposition}
    \fm= \fm_1\oplus\fm_2 \oplus\fm_3.
\end{equation}
The subalgebra $\so(2)\oplus\so(3)=\Span\{e_1,e_8,e_9,e_{10}\}\subset\so(5)$ provides a reductive decomposition $\so(5)=\so(2)\oplus\so(3)\oplus\fp$ (where $\fp=\fm_2\oplus\fm_3$) for the symmetric space 
\begin{equation}\label{eq:grasmanian}
    G_{5,2}=\frac{\SO(5)}{\SO(2)\times\SO(3)}.
\end{equation}
Furthermore, $\fp\simeq \C^3$ with respect to the canonical almost complex structure
$J=\begin{pmatrix}0&-I_3 \\I_3&0\end{pmatrix}$
and it turns out that   $\fm_3=\Span\{Je_2,Je_3,Je_4\}$. With respect to the basis \eqref{eq: so(5)_basis}, we define the alternating forms on $\fp$ given by
\begin{equation}\label{eq: su3-structure}
\begin{array}{rl}
\omega_0     &=e^{25}+e^{36}+e^{47},           \\[4pt]
\Im\Omega_0&=e^{237}-e^{246}+e^{345}-e^{567},\\[4pt]
\Re\Omega_0&=e^{234}-e^{267}+e^{357}-e^{456}.
\end{array}
\end{equation}
Notice that $\omega_0\wedge\Re\Omega_0=0$ and $2\omega_0^3=3\Re\Omega_0\wedge\Im\Omega_0\neq 0$, i.e. $(\omega_0,\Re\Omega_0)$ is a \emph{$\SU(3)$-structure} on $\fp$ (with induced orientation $-e^{234567}$). 
\begin{lemma}\label{lemma:su3-struct}
The forms in \eqref{eq: su3-structure} are $\ad(\so(2)\oplus\so(3))$-invariant, therefore, they induce a $\SO(5)$-invariant $\SU(3)$-structure $(\omega,\Re\Omega)$ on $G_{2}(\R^5)$. Furthermore, we have that $d\omega=0,$ $d(\Re\Omega)=3e^1\wedge\Im\Omega$ and $d(\Im\Omega)=-3e^1\wedge\Re\Omega.$ 
\end{lemma}
\begin{proof}
The Lie algebras  $\fsl(\C^3) $ and $ \fs\fp(\R^6)$ are  described by:\\
$$
\fsl(\C^3)=\left\{\begin{pmatrix}B&-C\\C&B\end{pmatrix}\in\gl(\fp) \Big\vert   B,C\in\fsl(\R^3)\right\},
\quad
\fs\fp(\R^6)=\left\{ \begin{pmatrix}B&C\\D&-B^T\end{pmatrix}\in\gl(\fp) \Big\vert  C,D\in\sym(\R^3)\right\}.
$$ 
The Lie algebras $\fsl(\C^3)$ and $\fs\fp(\R^6)$  can be interpreted as stabilisers 
$$ 
\frak{stab}_{\gl(\fp)}(\omega_0)=\fsl(\C^3)
$$  
and
$$  
\frak{stab}_{\gl(\fp)}(\Re\Omega_0)=\frak{stab}_{\gl(\fp)}(\Im\Omega_0)=\fraksp(\R^6).
$$  
Notice that the endomorphisms $\ad(e_8)_\fp$, $\ad(e_9)_\fp$, $\ad(e_{10})_\fp$ and 
$\ad(e_1)_\fp=J$
belong to the intersection  $\fsl(\C^3)\cap\fraksp(\R^6)$. Therefore, the forms in \eqref{eq: su3-structure} are $\ad(\so(2)\oplus\so(3))$-invariant. Now, for any $p=x\cdot\Big(\SO(2)\times\SO(3)\Big)\in G_{5,2}$ define $(\omega,\real\Omega)$ by
$$ 
    \omega_p=\rL_x^\ast\omega_0 \qandq \real\Omega_p=\rL_x^\ast\real\Omega_0,
$$
where $\rL_x$ is the left multiplication by $x\in \SO(5)$, hence $2\omega^3=3\Re\Omega\wedge\Im\Omega\neq 0$ . Finally, from Table \ref{tab:bracket} follows that $d\omega_0=0,$ $d(\Re\Omega_0)=3e^1\wedge\Im\Omega_0$ and $d(\Im\Omega_0)=-3e^1\wedge\Re\Omega_0$. 
\end{proof}  
The following result is mentioned in \cite{kerr1998new}*{\S 4}. For the sake of completeness,  we include the detailed proof.
\begin{lemma}\label{lem:invMetrics}
Any $\Ad(\SO(3))$-invariant metric on $\V$  is isometric to a diagonal one, i.e., up to isometries,  it is  enough to consider the class of invariant  metrics  determined by the $\Ad(\SO(3))$-invariant inner product: 
\begin{equation}\label{eq:sasakianmetric} g=y_1Q\vert_{\fm_1}+y_2Q\vert_{\fm_2}+y_3Q\vert_{\fm_3},\qwhereq y_i>0\quad i=1,2,3.    
\end{equation}      
\end{lemma}
\begin{proof}
In the decomposition \eqref{eq: irreducible so(3)-decomposition} we have $\fm_2\simeq\fm_3$, hence,   any $\SO(5)$-invariant  metric on $\V$ is determined by the $\Ad(\SO(3))$-invariant inner product $\langle\cdot,\cdot\rangle=Q(g\cdot,\cdot)$, where $g$ is a positive symmetric matrix of the form
$$ 
g=\begin{pmatrix}
 y_1 &0         & 0       \\
   0 & y_2I_{3} & xI_{3}  \\
   0 & xI_{3}   & y_3I_{3}    
\end{pmatrix} 
\qforq y_1,y_2,y_3>0, \quad x\in \R \qwithq y_2y_3>x^2.
$$
Denote by  $N(\SO(3))$  the normalizer of $\SO(3)$ in $\SO(5)$, since the quotient $\frac{N(\SO(3))}{\SO(3)}$  is isomorphic to $ \SO(2)$ the conjugation induces a 1-parameter family of diffeomorphisms on $\fm$ given by $\phi(t)=\Ad(\exp(te_1))$. Thus, for $t\in \R$ the one-parameter  family 
$ 
g_t=\phi(t)^T g \phi(t)  
$ 
defines an isometric $\Ad(\SO(3))$-inner product. 

A direct computation shows that the parameter $x(t)$ in the metric $g_t$, is given by
$x(t)=\frac{(y_3-y_2)}{2}\sin(2t)+x\cos(2t)$. Hence,  $g_t$ is diagonal if: 
$$
y_2\neq y_3 \qandq t=\frac12\arctan\left(\frac{2x}{y_2-y_3}\right)
$$
or
$$
y_2=y_3 \qandq t=\frac{(2k+1)\pi}{4} \qwithq k\in \N,
$$
\end{proof}
With respect to \eqref{eq:sasakianmetric}, we define an orthonormal basis of $\fm$:
\begin{equation}\label{eq: on-basis_so5}
X_1:=\frac{1}{\sqrt{y_1}}e_1, \quad X_i:=\frac{1}{\sqrt{y_2}}e_i \qandq X_{i+3}:=\frac{1}{\sqrt{y_3}}e_{i+3}\qforq  i=1,2,3.
\end{equation} 
Then, the Lie bracket in  Table \ref{tab:bracket} becomes
\begin{equation}\label{eq:brack}
\begin{split}
 [X_1,X_i]_{\fm} =-\sqrt{\frac{y_3}{y_1y_2}}X_{i+3}, \quad
 [X_1,X_{i+3}]_{\fm} =\sqrt{\frac{y_2}{y_1y_3}}X_i, \quad
[X_{i},X_{i+3}]_{\fm}=-\sqrt{\frac{y_1}{y_2y_3}}X_1.   
\end{split} 
\end{equation}
In \cite{reidegeld2010spaces}, Reidegeld proved that the space of $\SO(5)$-invariant $\gt$-structures on $\V$ has dimension $5$. The next Lemma describes the invariant $\gt$-structures on the Stiefel manifold up to \emph{$\gt$-equivalence}, i.e. for a $\gt$-structure $\varphi$ written in the basis $\{e_1,...,e_7\}$ and $\{\Tilde{e}_1,...,\Tilde{e}_7\}$ of $\fp$, there is a change of basis $f\in \gt$ from $\{e_1,...,e_7\}$ to $\{\Tilde{e}_1,...,\Tilde{e}_7\}$.
\begin{lemma}\label{lem: SO(5) invariant G2 structures}
The $5$-dimensional family of $\SO(5)$-invariant $\gt$-structures on $\V$ is described, up to $\gt$-equivalence,  by $(a,b,x,y,z)\in \R^5$ satisfying 
\begin{equation}\label{eq: open G2 condition}
     a^2+by>0, \quad b^2+ax>0 \quad z\neq 0 \qandq xy=ab.
\end{equation}
Explicitly, the $\gt$-structure $\varphi$ can be written as
\begin{align}\label{eq: on G2-form}
\varphi=&
-X^1\wedge(X^{25}+X^{36}+X^{47})+p_1X^{234}+p_2(-X^{267}+X^{357}-X^{456})\\ &+p_3(X^{237}-X^{246}+X^{345})-p_4X^{567},\nonumber
\end{align}
where
\begin{equation}
p_1=\frac{x\sqrt{a^2+by}}{b^2+ax}, \quad p_2=\frac{a}{\sqrt{a^2+by}}, \quad p_3=\frac{b}{\sqrt{b^2+ax}} \qandq p_4=\frac{y\sqrt{b^2+ax}}{a^2+by}
\end{equation}
     and
     \begin{align}\label{eq: G2 metric coefficients}
         y_1=\frac{z^2}{y_2y_3}, \quad y_2=\frac{(b^2+ax)^{2/3}}{(a^2+by)^{1/3}} \qandq y_3=\frac{(a^2+by)^{2/3}}{(b^2+ax)^{1/3}}.
     \end{align}
 \end{lemma}
\begin{proof}
From \eqref{eq: su3-structure} it  is straightforward to check  that any $\Ad(\SO(3))$-invariant $3$-form on $\fm$ is given by 
\begin{align*}
\wt\varphi=&-\wt ze^1\wedge(e^{25}+e^{36}+e^{47})+\wt x e^{234}+\wt  a(-e^{267}+e^{357}-e^{456})\\ 
&+\wt b(e^{237}-e^{246}+e^{345})-\wt y e^{567}.\nonumber
\end{align*}
The $3$-form $\wt\varphi$ defines a $\gt$-structure if the induced symmetric bilinear  form  
\begin{equation*}
B_{ij}=\Big((e_i\lrcorner\wt{\varphi})\wedge(e_j\lrcorner\wt{\varphi})\wedge\wt{\varphi}\Big)(e_1,...,e_{7}) 
\end{equation*}
is positive (or negative) definite. 
Explicitly, we obtain that 
$ 
B_{11}=6\wt z^3 
$ 
and for  $k=2,3,4$ 
$$B_{kk}=6\wt z(\wt{b}^2+\wt a\wt x), \quad B_{k+3 k+3}=6\wt{z}(\wt{a}^2+\wt{b}\wt{x}), \quad B_{k k+3}=3\wt{z}(\wt{a}\wt{b}-\wt{x}\wt{y})$$
and $B_{ ij}=0$ otherwise. 

Now, since $\phi(t)=\Ad(\exp(te_1))\in \SU(3)\subset \gt$ [cf. Lemma \ref{lem:invMetrics}], there exist $t_0\in \R$ such that 
\begin{align*}
    \varphi=\phi(t_0)^*\widetilde{\varphi}=&-z e^1\wedge(e^{25}+e^{36}+e^{47})+xe^{234}+a(-e^{267}+e^{357}-e^{456})\\ &+b(e^{237}-e^{246}+e^{345})-y e^{567},\nonumber
\end{align*}
satisfies \eqref{eq: open G2 condition}. On the other hand, the induced metric is given by
$$
g=\frac{B}{6^{\frac29}\det(B)^{\frac19}},
$$  
hence, by comparing $g$ with \eqref{eq:sasakianmetric}, we obtain the conditions in \eqref{eq: G2 metric coefficients}. Finally, the expression in \eqref{eq: on G2-form} is obtained by  writing $\varphi$ in the orthonormal basis \eqref{eq: on-basis_so5}.
 \end{proof}
\begin{definition}
A $\gt$-structure $\varphi$ is called \emph{coclosed} if $d\psi=0$. Among the class of coclosed class, $\varphi$ is said \emph{purely coclosed}  if $d\varphi\wedge\varphi=0$ and  \emph{nearly parallel} if $d\varphi=4\psi$.
\end{definition}
From Lemma \eqref{lem: SO(5) invariant G2 structures}, the following result describes all coclosed invariant $\gt$--metrics  on $\V$
\begin{theorem}\label{thm:Nearly}
The $\gt$-structure \eqref{eq: on G2-form} is coclosed for any    $(a,b,x,y,z)\in \R^5$ satisfying \eqref{eq: open G2 condition}. Moreover, $\V$ does not carry a $\SO(5)$-invariant purely coclosed $\gt$-structure  and \eqref{eq: on G2-form} is nearly parallel if 
\begin{align}\label{eq: nearly parallel condition}
    x=a, \quad y=b, \quad a^2+b^2=\frac{27}{512} \qandq z=-\frac{9}{32}.
\end{align}
\end{theorem}

\begin{proof}
From Lemma \ref{lem: SO(5) invariant G2 structures}, we get the dual 4-form $\psi=\ast\varphi$ 
\begin{align}\label{eq: 4-form psi}
\psi=&
-(X^{2356}+X^{2457}+X^{3467})-p_4X^{1234}+p_3X^1\wedge(X^{267}-X^{357}+X^{456})\\ &+p_2X^1\wedge(X^{237}-X^{246}+X^{345})-p_1X^{1567}.\nonumber
\end{align}
Using the Lie bracket \eqref{eq:brack} and the metric coefficients \eqref{eq:sasakianmetric}, we obtain the exterior derivative for the following $\Ad(\SO(3))$-forms on $\fm$:
\begin{align}\label{eq:exteriorDerivatives}
    dX^1=& \frac{z}{y_2y_3}\left(X^{25}+X^{36}+X^{47}\right)\\  \nonumber
    dX^{234}=&-\frac{y_2}{z}\left(X^{1237}-X^{1246}+X^{1345}\right), \quad 
    dX^{567}=\frac{y_3}{z}\left(X^{1267}-X^{1357}+X^{1456}\right)\\ \nonumber
    d(X^{267}&-X^{357}+X^{456})=\frac{2y_3}{z}\left(X^{1237}-X^{1246}+X^{1345}\right)-\frac{3y_2}{z}X^{1567}\\ \nonumber
    d(X^{237}&-X^{246}+X^{345})=\frac{3y_3}{z}X^{1234}-\frac{2y_2}{z}\left(X^{1267}-X^{1357}+X^{1456}\right).
\end{align}
From \eqref{eq: 4-form psi} and \eqref{eq:exteriorDerivatives}, it is easy to check that $d\psi=0$. Now, from \eqref{eq: on G2-form} we obtain
\begin{align}\label{eq: derivative varphi}
d\varphi=&\frac{2z}{y_2y_3}\left(X^{2356}+X^{2457}+X^{3467}\right)-\frac{p_1y_2+2p_2y_3}{z}\left(X^{1237}-X^{1246}+X^{1345}\right)\\ \nonumber
&+\frac{3p_2y_2}{z}X^{1567}+\frac{3p_3y_3}{z}X^{ 1234}-\frac{p_4y_3+2p_3y_2}{z}\left(X^{1267}-X^{1357}+X^{1456}\right).
\end{align}
From a direct  computation, we note that
$
\ast\left(d\varphi\wedge\varphi\right)=12+\frac{3}{y_2}+\frac{3}{y_3}\neq 0,
$ 
in other words,  $\varphi$ given in  \eqref{eq: on G2-form} is not purely coclosed. Finally, by imposing the condition $d\varphi=4\psi$, and using \eqref{eq: 4-form psi} and \eqref{eq: derivative varphi}, we obtain \eqref{eq: nearly parallel condition}.
\end{proof}

\begin{remark}
    The $S^1$-family given in \eqref{eq: nearly parallel condition} corresponds to the general construction given in \cite{alexandrov2012}*{(4.28)}.
\end{remark}
\subsection{Homogeneous Sasakian metrics on $\V$}\label{sec:Sasakian}
In this section, we study the Sasakian condition for $\SO(5)$-invariant metrics on $\V$. Firstly, we characterise the class  $(\eta,\xi,g,\Phi)$ of  contact metric structures, cf. Section \ref{sec:intro}. Secondly, in  Theorem \ref{thm:V52 Sasakian1}, we provide the constraints in the metric to get a Sasakian structures. And finally, we relate the invariant $\gt$-structures inducing a Sasakian metric. 

\begin{lemma}\label{lem:contactmetricGeneral}
The globally defined $1$-form $\eta=X^1$ and $\xi=X_1$ define  a $\SO(5)$-invariant almost contact structure on $\V$. Furthermore,  $(\eta,\xi,g,\Phi)$ defines a contact metric structure on $\V$, if and only if 
\begin{align}\label{eq:contactmetric}
g=4y_2y_3Q\vert_{\fm_1}+y_2Q\vert_{\fm_2}+y_3Q\vert_{\fm_3} \qandq \Phi(X)=[e_1,X].
\end{align}
\end{lemma}

\begin{proof}
First, note that $(\V,\eta)$ is a contact manifold, since it is clear that $\eta(\xi)=1$ and that  $\eta\wedge d\eta^3\neq 0$ follows from \eqref{eq:exteriorDerivatives}. Now, we define an $\Ad(\SO(3))$-invariant contact metric structure $(\eta,\xi,g,\Phi)$, i.e., we define an $\Ad(\SO(3))$-invariant endomorphism $\Phi$ on $\fm$ such that $\Phi^2=-I+\eta\otimes\xi$ and $d\eta(X,Y)=2g(X,\Phi(Y))$. From \eqref{eq:brack} we have
$$
d\eta=\sqrt{\frac{y_1}{y_2y_3}}\left(X^{25}+X^{36}+X^{47}\right),
$$
Hence, the condition  $d\eta(X_i,X_{i+3})=2g(X_i,\Phi(X_{i+3}))$ implies that $\Phi$ is given by
\begin{equation}\label{eq:nuevoPhi}
\Phi X_i=-\frac{1}{2}\sqrt{\frac{y_1}{y_2y_3}}X_{i+3}, \quad
\Phi X_{i+3}=\frac{1}{2}\sqrt{\frac{y_1}{y_2y_3}}X_i,\qandq \Phi\xi=0.
\end{equation}
Moreover, using the above equalities and  the condition $\Phi^2=-I+\eta\otimes\xi$, we obtain that $-\frac14\frac{y_1}{y_2y_3}=-1$, or equivalently $y_1=4y_2y_3$.  
\end{proof}

The Nijenhuis  tensor induced by $\Phi$ is defined by
\begin{equation}\label{eq:Nijenhuis}
N_\Phi(X,Y)=[\Phi X,\Phi Y]+\Phi^2[X,Y]-\Phi[\Phi X,Y]-\Phi[X,\Phi Y]. 
\end{equation}
From  \cite{Boyer2008}*{Theorem 6.5.9}, we know  that  an almost contact structure $(\xi,\eta,\Phi)$ is \emph{normal} if the Nijenhuis tensor  
satisfies \begin{equation}\label{eq:criterioN}
N_{\Phi}=-d\eta\otimes \xi.    
\end{equation} 
Following \cite{Boyer2008}*{Definition 6.5.13}, we define:

\begin{definition}
    A normal contact metric structure $\cS=(\eta,\xi,g,\Phi)$ on $M^{2n+1}$ is called Sasakian and $(M^{2n+1},\cS)$ is called a \emph{Sasakian manifold}.
\end{definition}

\begin{theorem}\label{thm:V52 Sasakian1} 
The $\SO(5)$-invariant contact metric structure  $(\eta,\xi,g,\Phi)$ given in Lemma \ref{lem:contactmetricGeneral} is a Sasakian structure, if and only if, the metric  \eqref{eq:contactmetric} satisfies 
\begin{equation}\label{eq:SasakiConditionY_i}
y_1=4y_2^2 \qandq y_2=y_3.  \end{equation} 
\end{theorem} 
\begin{proof} 
We show that  $(\eta,\xi,g,\Phi)$ is a normal, contact metric structure, i.e., Sasakian. Using the bracket \eqref{eq:brack} and the definition of  $\Phi$ in \eqref{eq:contactmetric},  we verify that the Nijenhuis tensor \eqref{eq:Nijenhuis} satisfies \eqref{eq:criterioN}. From Lemma \ref{lem:contactmetricGeneral}, we recall that the metric $g$ is contact for $y_1=4y_2y_3$. 
First, note that
$ 
N_\Phi(\xi,Y)= \Phi^2[\xi,Y] -\Phi[\xi,\Phi Y]
$ and now, for $i=1,2,3$ we consider 
\begin{align*}
\Phi^2[\xi,X_{i}] -\Phi[\xi,\Phi X_{i}]&=-[\xi,X_{i}] -\Phi[\xi,-X_{i+3}]
=-[\xi,X_{i}]+\sqrt{\frac{y_2}{y_1y_3}}\Phi X_{i}\\ 
&=\sqrt{\frac{y_3}{y_1y_2}}X_{i+3}-\sqrt{\frac{y_2}{y_1y_3}} X_{i+3}\\
&=\left(\frac{1}{2y_2}-\frac{1}{2y_3}\right)X_{i+3}
\end{align*}
Since  $i_\xi d\eta=0$, we conclude that $y_2=y_3$. Analogously 
\begin{align*}
\Phi^2[\xi,X_{i+3}] -\Phi[\xi,\Phi X_{i+3}]&=-[\xi,X_{i+3}] -\Phi[\xi,X_{i}]
=\sqrt{\frac{y_2}{y_1y_3}} X_{i}-\sqrt{\frac{y_3}{y_1y_2}}\Phi X_{i+3}\\ 
&=\frac{1}{\sqrt{y_1}} X_{i}-\frac{1}{\sqrt{y_1}}X_{i}=0.
\end{align*}
For the missing cases,  for $X,Y\in \fm_2\oplus\fm_3$ the brackets  $[X,Y]$, $[\Phi Y,Y]$ and $[X,\Phi Y]$ are generated by  $\xi$ (cf. \eqref{eq:brack}), hence  $\Phi[X,Y]=\Phi[\Phi X,Y]=\Phi[X,\Phi Y]=0$ and thus $N_\Phi(X,Y)=[\Phi X,\Phi Y]$. So, \eqref{eq:criterioN} holds by noting that 
$$
N_\Phi(X_i,X_{i+3})=-2X_1 \qforq i=2,3,4 \qandq N_\Phi(X_i,X_j)=0 \quad \text{otherwise}.
$$
\end{proof}

According to the notation in Theorem \ref{thm:V52 Sasakian1}, the Sasakian metric  \eqref{eq:sasakianmetric} is induced by the following family of $\gt$-structures.   
\begin{corollary}\label{cor:V52 Sasakian} 
Consider the Sasakian metric on $\V$ as a metric induced by \eqref{eq: on G2-form} and the parameter vector $(z,a,b,x,y)\in \R^5$ satisfying the condition  \eqref{eq: open G2 condition}. Then, the $\gt$-structure \eqref{eq: on G2-form} induces a Sasakian metric, if and only if,
\begin{equation}\label{eq: G2 Sasakian metric}
 x=a, \quad y=b \qandq z^2=4(a^2+b^2)^{\frac43}.
 \end{equation}
In this case, the $2$-parameter family of coclosed $\gt$-structures is
\begin{align*}
\varphi=&
-X^1\wedge(X^{25}+X^{36}+X^{47})+\left(\frac{2}{z}\right)^{\frac34}\Big\{a\Big(X^{234}-X^{267}+X^{357}-X^{456}\Big) \\ 
& +b\Big(X^{237}-X^{246}+X^{345}-X^{567}\Big)\Big\},\\
\psi=&
-\Big(X^{2356}+X^{2457}+X^{3467}\Big)-\left(\frac2z\right)^{\frac34}\Big\{bX^1\wedge\Big(X^{234}-X^{267}+X^{357}-X^{456}\Big) \\ 
&-aX^1\wedge\Big(X^{237}-X^{246}+X^{345}-X^{567}\Big)\Big\}.\nonumber
\end{align*}
\end{corollary}
\begin{proof} 
We compare \eqref{eq: G2 metric coefficients} and  \eqref{eq:SasakiConditionY_i}, thus 
$$
z^2=4(a^2+by)^{4/3} \qandq a^2+by=b^2+ax>0,
$$
since $a^2b^2=abxy$, we have $a^2(ax+b^2)=(b^2+ax)ax$ and $by(a^2+by)=b^2(by+a^2)$. Thus, the result follows.  
 
\end{proof}

\begin{remark}
On the $\ad(\so(3))$-invariant reductive decomposition $\fm=\fm_1\oplus\fm_2\oplus\fm_3$ \eqref{eq: reductive decomposition so(5)} the tensors given by
\begin{align*}
\xi=\frac{4}{3}e_{1}, \quad \Phi=
\begin{pmatrix}0&&\\&& I_3\\ &-I_3& \end{pmatrix} \qandq 
g=\frac38\begin{pmatrix}\frac{3}{8} &0\\ 0& I_6\end{pmatrix},
\end{align*}
define  the $\SO(5)$-invariant Sasakian-Einstein structure on $\V$, associated with nearly parallel $\gt$-structure of type II (cf.\cite{friedrich1997nearly}*{Table 2}). 
\end{remark}

\section{$\gt$ and self dual contact instantons on $\V$}\label{sec:InstantonCOnstruction}
According to Corollary \ref{cor:V52 Sasakian}, we can consider solutions of \eqref{eq:introInstEquation} for those $\gt$-structures with induced Sasakian metric. 
Additionally, if we endow a complex vector bundle $E\to\V$ with a holomorphic Sasakian structure \cite{biswas2010vector}, one obtain the notion of \emph{Chern connection}, as those connections simultaneously compatible with the holomorphic structure and some Hermitian bundle metric. The  Hermitian Yang-Mills (HYM) condition is naturally extended to a Sasakian version, namely, let $\omega\colon=d\eta\in \Omega^{1,1}(M)$ be the transversal Kähler form, we say that a connection $A\in \cA(E)$ is \emph{transverse Hermitian Yang-Mills} connection (tHYM) if
\begin{equation}\label{eq:tHYM}
\hat{F}_A\colon = \langle F_A,\omega\rangle=0, \qandq  F^{0,2}_A = 0.
\end{equation}
Here, the type $(p,q)$ of a  $k$-form on $M$  is induced by the condition $\Phi\vert_{H}^2=-1$,  where $H$ is the Kähler distribution given by $H=\ker\eta\subset TM=H\oplus \xi$. Additionally, 
the splitting on the tangent bundle induced by the contact structure yields the splitting  $\Omega^k(M)=\Omega^k(H^\ast)\oplus\eta\wedge\Omega^{k-1}(H^\ast)$. Hence, after complexification of the tangent bundle, one can naturally extend 
this bi-degree   $(p,q)$   to the bundle of $E$-valued $k$-forms.\\

We know that   $2$-forms on $M$, and consequently $\Omega^2(\fg_E)$, decomposes as (cf. \cite[Section.~2.1]{portilla2023}) 
\begin{equation}\label{eq:2formsDecomposition}
\Omega^2(M) =\underbrace{\Omega^2_1 \oplus\Omega^2_6\oplus \Omega^2_8}_{\text{Transversal $2$-forms}} \oplus \Omega^2_V  
\end{equation}
Furthermore, a connection $A\in\cA(E)$ is  SDCI  if its  curvature belong to $\Omega^2_8(\fg_E)$, or equivalently, if the curvature $F_A\in\Omega^2(\fg_E)$ is of  type $(1,1)$ orthogonal to $\omega$ (cf. \cite[Proposition~2.5]{portilla2023}).  

\begin{lemma}\label{lem:sdci g2}
With respect to the Sasakian structure of $\V$ given in Theorem \ref{thm:V52 Sasakian1} and the $\gt$-structure in  Corollary \ref{cor:V52 Sasakian} we have that: If $A$ is a SDCI, then it is a $\gt$-instanton.    
\end{lemma}
\begin{proof}
We know that in local coordinates, a  basis for self dual contact $2$--forms is given by (cf. \cite[Section~2]{portilla2023})
\begin{equation}\label{eq:sdci basis}
\begin{array}{lccl} 
w_1=X^{23}+X^{56}, &w_2=X^{26}+X^{35}, &w_3=X^{24}+X^{57}, &w_4=X^{27}+X^{45},\\[4pt]  
w_5=X^{34}+X^{67}, &w_6=X^{37}+X^{46}, &w_7=X^{25}-X^{47}, &w_8=X^{36}-X^{47}. 
\end{array}    
\end{equation}
For $\Re\Omega_0$   
 and $\Im\Omega_0$   
 given in Lemma \ref{lemma:su3-struct}, one verifies that $w_i\wedge\real\Omega_0=w_i\wedge\Im\Omega_0=0$ for each $i=1,\cdots,8$ in \eqref{eq:sdci basis}. Hence,  if $A$ is a SDCI, then its curvature $F_A$ is locally  generated by $w_i$'s in \eqref{eq:sdci basis}. Thus, for the $\gt$-structure (cf. 
Corollary \ref{cor:V52 Sasakian})
$$ 
\varphi= -\eta\wedge d\eta
+\left(\frac{2}{z}\right)^{\frac34}a\Re(\Omega_0)
+\left(\frac{2}{z}\right)^{\frac34}b\Im(\Omega_0),
$$
we have
\begin{align*}
\ast(\varphi\wedge F_A)&=
-\ast(\eta\wedge d\eta\wedge F_A)+\left(\frac2z\right)^{\frac34}a\ast(\Re(\Omega_0)\wedge F_A)+\left(\frac2z\right)^{\frac34}b\ast(\Im(\Omega_0)\wedge F_A)\\
&=-\ast(\eta\wedge d\eta\wedge F_A)=-F_A.
\end{align*}
\end{proof}

\begin{remark}\label{rem:SDCIG2}
We can generalise Lemma \ref{lem:sdci g2} in the following sense. Observe that from the decomposition in \eqref{eq:2formsDecomposition},  
we have the $\SU(3)$-invariant decomposition
\begin{equation}\label{eq:omega21}
\Omega^2(H)=\Omega^2_6\oplus\Omega^2_1\oplus\Omega_8^2,    
\end{equation}

where $\Omega^2_8\simeq\su(3)$ and $\Omega^2_1=\Omega^0(M)\otimes\omega$. On the other hand, we know that the space of $2$-forms on $H^\C=H\otimes_\R \C$ decomposes into the $\Gl(3,\C)$-invariant submodules 
\begin{equation*}
\Omega^2(H^\C)=\Omega^{2,0}(H^\C)\oplus\Omega^{0,2}(H^\C)\oplus \Omega^{1,1}(H^\C).
\end{equation*}
By Schur's Lemma, we have  (see \cite{falcitelli1994}*{Equation (1.2)}
\begin{equation}
\Omega^{2,0}\oplus\Omega^{0,2}\simeq \Omega^2_6\otimes_\R\C\qandq \Omega^{1,1}\simeq (\Omega^2_1\oplus\Omega^2_8)\otimes_\R\C.
\end{equation}
On the other hand, we have the $\gt$-invariant decomposition
$  
\Omega^2(TM)=\Omega^2_7\oplus \Omega^2_{14},
$ 
where $\Omega^2_{14}\simeq\fg_2$. Since $\SU(3)\subset \gt$ we have that $\Omega_8^2\subset\Omega_{14}^2$ and $\Omega^2_1\subset \Omega^2_7$.
It is easy to see that $\Omega^2_V$ has elements in both $\Omega_7^2$ and $\Omega_{14}^2$. Hence, the same is true for   $\Omega^2_6$.  
\end{remark}
\begin{lemma}\label{lem:g2 sdci}
Let $E\to \V$ be a Sasakian holomorphic vector bundle \cite[Section~3]{Boyer2008},  the Sasakian structure $(\eta,\xi,g,\Phi)$ is given in Theorem  \ref{thm:V52 Sasakian1}  and $A\in\cA(E)$ the Chern connection. If  $A\in \cA(E)$ is a $\gt$-instanton, then $A$ satisfies the equation \eqref{eq:introInstEquation}, i.e. it is a  $\lambda$-contact instanton . 
\end{lemma}
\begin{proof}
Consider  $\varphi$ given in  Corollary \ref{cor:V52 Sasakian} and set  $dz^k=X^k+\sqrt{-1}X^{k+3}$ for $k=2,3,4$. From  Lemma \ref{lemma:su3-struct} and \eqref{eq: su3-structure}, we know that  $\Omega_0=dz^2\wedge dz^3\wedge dz^4\in \Lambda^{3,0}H^\ast_\C$, where $H:=\ker(\eta)$. Now, if $A\in\cA(E)$ is the Chern connection, then  $F_A$ is of $(1,1)$-type \cite[p.552]{biswas2010vector}. By taking into consideration the bi-degree of $\Omega_0$, we conclude that $F_A\wedge\Omega_0=0$ and  $F_A\wedge\overline{\Omega_0}=0$, hence both $F_A\wedge \Im{\Omega_0} $ and $F_A\wedge \Re{\Omega_0}$ vanishes, then 
\begin{align*}
F_A\wedge\varphi
&=F_A\wedge\left(\eta\wedge d\eta
+\left(\frac{2}{z}\right)^{\frac34}a\Re\Omega_0
+\left(\frac{2}{z}\right)^{\frac34}b\Im\Omega_0\right)\\
&=F_A\wedge\eta\wedge d\eta. 
\end{align*} 
Consequently, solutions to the $\gt$-instanton equation are also solutions to the contact instanton equation.  
\end{proof} 
Consider the  Kähler manifold  $G_{2}(\R^5)\cong\V/S^1$, we fix a holomorphic bundle $(\cE,\bar{\partial})$  on $ G_{5,2}$,   the pullback $\cE^\ast:=\pi^\ast \cE \to \V $ carries a structure of  Sasakian holomorphic bundle 
 \[
\begin{diagram}
\node{\cE^\ast} \arrow{s,l}{}
\node{\cE} \arrow{s,r}{} \\
\node{\V} \arrow{e,b}{\pi} \node{G_{2}(\R^5)}
\end{diagram}
\]
where  $\pi \colon \V \to G_{2}(\R^5)$ is  a Sasakian circle fibration endowed with the $\gt$-structure \eqref{eq: on G2-form}. Finally, in \cite[Lemma~2.4]{portilla2023} was proved that SDCI are precisely  connections whose curvature is of type $(1,1)$ and orthogonal to $\omega$, %
(cf. the splitting in \eqref{eq:2formsDecomposition} and Remark \ref{rem:SDCIG2}). Thus, from Lemmas \ref{lem:sdci g2} and  \ref{lem:g2 sdci}, we obtain the following result:
\begin{corollary}\label{cor:tHYM}
 The vector bundle  $\cE^\ast$ is a holomorphic Sasakian bundle, and a Chern connection $A$ on $\cE$ is HYM if, and only if, 
$\pi^\ast A$ is a $\gt$-instanton on $\cE^*$.
\end{corollary} 
We recall that in Lemma \ref{lem:sdci g2},  Lemma \ref{lem:g2 sdci} and Corollary \ref{cor:tHYM} the key argument  follow just in presence of a SU$(n)$-structure.  These have analogous results  in \cite[Proposition~3.11]{portilla2023} and  \cite[Lemma~21]{calvo2020gauge}, the difference in our  case is  that we do not assume $M$ is a \emph{contact Calabi-Yau}  but only endowed with a transversal $\SU(3)$-structure. 
\section{Homogeneous Principal Bundles on $\V$}\label{sec:homogeneousBundles}
First, we recall some notation and results about homogeneous principal bundles, following \cite{cap2009}*{Section 1.4}. Let $H\subset G$ be  a closed subgroup  of the Lie group $G$, we denote by  $M=G/H$ the homogeneous space, let $L\colon G\times M\to M$ be the canonical left action and $p\colon G\to M$ an $H$--principal bundle  given by the canonical projection. 
\begin{definition}
A \textit{homogeneous principal bundle } over $M$ is a locally trivial principal bundle $\pi\colon P\to M$ together with a left $G$--action $\wt{L} \colon G\times P\to P$, which lifts the action $L$ on $M$, i.e., $\pi(\wt{L}(\ig,y))=L(\ig,\pi(y))$, and for each $\ig\in G$, the bundle map $\wt{L}_\ig\colon P\to P $ is a homomorphism of  principal bundles.    
\end{definition} 
Homogeneous principal bundles are characterised in the following way:  
\begin{lemma}\cite[Lemma~1.4.5]{cap2009}  
\label{lem:linearMapCondition}
Let $G$ and  $K$ be Lie groups and $H\subset G$ a closed Lie subgroup. Let $P\to G/H$ be  a homogeneous   $K$-principal bundle, then   there exists a  smooth  homomorphism  $\phi\colon H\to K$ such that $P\cong G\times_\phi K$. The action of $H$ on $K$ is given by $h\cdot k=\phi(h)\cdot k$ for each $h\in H$ and $k\in K$. The bundles  corresponding  to two homomorphisms $\phi, \wt{\phi}\colon H\to K$ are isomorphic if and only if  
there exists $k\in K$ such that  $k\cdot\phi=\wt\phi\cdot k$     
\end{lemma}

For a homogeneous $K$-principal bundle $P\to G/H$, a connection $A\in \Omega^1(P,\fk)$ is called \emph{invariant} if $\wt{L}_\ig^*A=A$ for each $\ig\in G$.  

\begin{lemma}\label{rem:linearMapCondition}\cite[Theorem 1.4.5]{cap2009} Let $P=G\times_\phi K\to G/H$ be the homogeneous $K$-principal bundle induced by the homomorphism $\phi: H\to K$. Then, invariant principal connections on $P$ are in bijective correspondence with linear maps $\alpha\colon \fg\to \fk$ such that 
\begin{itemize} 
\item[(i)]
$\alpha\vert_{\fh}=\phi_\ast\colon \fh\to \fk$, the derivative of $\phi$.    
\item[(ii)] 
$\alpha\circ \Ad(h)= \Ad(\phi(h))\circ \alpha$ for all $h\in H$.
\end{itemize} 
\end{lemma}
\subsection{Homogeneous $\SO(3)$-principal bundles}\label{sec:SO(3)Bundles}
Consider $H=K=\SO(3)$ and $G=\SO(5)$. Since $\SO(3)$ is simple, up to conjugation there are   two smooth homomorphisms  $\phi_i\colon \SO(3)\to \SO(3),$ $i=0,1$. Namely, $\phi_0$  the trivial homomorphism and $\phi_1=1_{\SO(3)}$ the identity homomorphism, according  to Lemma \ref{lem:linearMapCondition}, there exist precisely two\footnote{Up to isomorphism of homogeneous principal bundles} homogeneous $\SO(3)$-principal bundles $\pi_i\colon P_i=G\times_{ \phi_i}\SO(3)\to\V$,  $i=0,1$. 
\begin{lemma}\label{lem:ASDCIonP0}
There is a unique invariant $\SO(3)$--connection on $P_0=\V\times \SO(3)$, given by
\begin{equation}\label{eq:trivialConnection}
\alpha=e^{1}\otimes(b_8e_8+b_9e_9+b_{10}e_{10}).   \end{equation}
Furthermore,  the curvature $F_\alpha$ of the connection $\alpha$ belongs to $\Omega^2_1(\so(3))$ cf. \eqref{eq:omega21} and Remark \ref{rem:SDCIG2}.  
\end{lemma}
\begin{proof}
Consider an arbitrary linear map $\alpha: \so(5)\rightarrow \so(3)$, in the  basis \eqref{eq: so(5)_basis} we write   
$$
\alpha=a_ie^i\otimes e_8+b_je^j\otimes e_9+c_ke^k\otimes e_{10}.
$$   
From Lemma \ref{rem:linearMapCondition},  $\alpha$ is an invariant connection, if and only if, it satisfies 
$\alpha\vert_{\so(3)}=0$ and  $\alpha\circ \Ad(\ig)=\alpha$, for all $\ig\in \SO(3)$. The  condition $\alpha\vert_{\so(3)}=0$  implies  that $\alpha$ is of  the form   $\alpha=a_ie^i\otimes e_8+b_je^j\otimes e_9+c_ke^k\otimes e_{10}$ for $i,j,k\in \{1,\dots, 7\}$. Now, we apply the second condition to 
$
\ig_k=
\begin{pmatrix} 
I_{2\times 2}&\\
0& h_k\\
\end{pmatrix}\in \SO(5),$ we define some values of $h_k$, 
\begin{equation}\label{eq:h}
h_0=\begin{pmatrix} 0& 1& 0\\1& 0& 0\\
0& 0&-1 \end{pmatrix}, 
h_1=\begin{pmatrix} 0& 1& 0\\-1& 0& 0\\0& 0&1 \end{pmatrix}, 
h_2=\begin{pmatrix} 0& 0& 1\\0&-1& 0\\1& 0& 0\end{pmatrix} \qandq
h_3=\begin{pmatrix} 0 & 0& 1\\0 & 1& 0\\-1& 0& 0\end{pmatrix}
\end{equation}
First, considering        
$h_0$ and $h_1$, we obtain that $\alpha(e_i)=a_ie_8+b_ie_9+c_ie_{10}$ and,  for $i\neq 3$ and $i\neq 6$  
\begin{align*}
\alpha(\Ad(\ig_0)(e_i))=\alpha(\ig_0e_i\ig_0^{-1})&=a_ie_8+b_ie_{10}+c_ie_9  \\
\alpha(\Ad(\ig_1)(e_i))=\alpha(\ig_1e_i\ig_1^{-1})&=-a_ie_8-b_ie_{10}-c_ie_9 
\end{align*}
it follows that $a_i=b_i=c_i=0$ for $i\in\{2,4,5,7\}$. 
Now, by considering 
$h_2$ and $h_3,
$
we obtain for $i=3,6$ 
$$\begin{array}{cl}
\alpha(\Ad(\ig_0)(e_i))=&-\alpha(e_i)=a_ie_8+b_ie_{10}+c_ie_9  \\ [5pt]
\alpha(\Ad(\ig_1)(e_i))=&\alpha(e_i)=-a_ie_8-b_ie_{10}-c_ie_9. 
\end{array}$$ 
Consequently, an invariant connection $\alpha$ is of the form
\begin{equation} 
\alpha=e^{1}\otimes(b_8e_8+b_9e_9+b_{10}e_{10}).   \end{equation} 
The curvature of $\alpha$  is  $F_\alpha=y_2^{-1}(X^{25}+X^{36}+X^{47})\otimes(b_8e_8+b_9e_9+b_{10}e_{10})$. The second statement  in the Lemma follows from 
\eqref{eq: su3-structure} since 
\begin{align}\label{eq: SO(3)_curvature_trivial}
F_\alpha
=2\omega\otimes(b_8e_8+b_9e_9+b_{10}e_{10}).
\end{align}
\end{proof} 
\begin{lemma}[Homogeneous $\U(1)$-connections]\label{rem:U(1)} 
There is a unique invariant $\U(1)$-connections on $P=\V\times \U(1)$, given by
\begin{align}\label{eq:U1connection}
    \beta=a_1e^1\otimes e_1
\end{align}
Furthermore,  the curvature $F_\alpha$ of the connection $\alpha$ belongs to $\Omega^2_1$ cf. \eqref{eq:omega21}.  
\end{lemma}
\begin{proof}
Similar to Lemma \ref{lem:ASDCIonP0}, since $\SO(3)$ is simple, all non-trivial homomorphism $\phi\colon \SO(3)\rightarrow K$ is either trivial or injective for all Lie group $K$. In particular for  $K=\U(1)$, from Lemma \ref{lem:linearMapCondition},  there exists a unique\footnote{Up to isomorphism of principal bundles} homogeneous $\U(1)$-bundle $P$ corresponding to $\phi$, namely, the trivial bundle  $P=\V\times \U(1)$.    
Any invariant connection on $P$ is given by a linear map $\beta: \so(5)\rightarrow \fu(1)$ satisfying  
$$
(i) \quad \beta\vert_{\mathfrak{so}(3)}=0=\phi_\ast\colon \so(3)\to \R \qandq (ii) \quad \beta\circ \Ad(a)= \beta, \text{ for all } a\in \SO(3).
$$
By considering $e_1$ the generator of  the Abelian Lie algebra $\fu(1)=\R$, the unique invariant connection is given by 
$$
\alpha=a_ie^i\otimes e_1 \qforq  i=1,\dots, 10.
$$
From item $(i)$,  since $\alpha\vert_{\so(3)}=0$, $\alpha$ has to be of the form  $\alpha=a_ie^i\otimes e_1$,  $i=1,\cdots, 7.$
On the other hand,  considering item $(ii)$ for $\ig_0,\ig_1$  given in \eqref{eq:h}, we use the same argument used in the proof of Lemma \ref{lem:ASDCIonP0} to obtain $\alpha(e_i) =a_ie_1$ and  for $i=2,\cdots,7$
\begin{align*}
\alpha(\Ad(\ig_0)(e_i))=\alpha(\ig_0e_i\ig_0^{-1}))&=-a_ie_1  \\
\alpha(\Ad(\ig_1)(e_i))=\alpha(\ig_1e_i\ig_1^{-1}))&=a_ie_1. 
\end{align*}
 Therefore, $\alpha$ is characterised by  
\begin{align}\label{eq:u1connection}
    \beta=a_1e^1\otimes e_1
\end{align}
and its curvature $2$-form is
\begin{equation}\label{eq: U(1)_curvature}
F_\beta=2(X^{25}+X^{36}+X^{47})\otimes e_1.
\end{equation}
Note that $F_\beta\in \Omega^2_1=\omega\otimes\Omega^0(M)$ \eqref{eq:split2forms}.
\end{proof}

Now, we consider the $\SO(3)$-principal bundle $P_1=\SO(5)\times_{\phi_1}\SO(3)$  induced by the homomorphism identity $\phi_1\colon \SO(3)\to \SO(3)$. The projection of $P_1$ is  given by 
\begin{equation}\label{eq:pi_1homogeneous_bundle}
 \pi_1\colon  [(u,s)]\in P_1\mapsto u\cdot\SO(3)\in \V, 
\end{equation}
and the class is given by  $[(u,s)]=\{(u\cdot r, r^{-1}s)\; \vert \;r\in \SO(3)\}.$ 
There exist a left $\SO(5)$-action on $P_1$ such that for each $v\in \SO(5)$ the bundle 
map $L_v\colon P_1\to P_1$  is $\pi_1$-equivariant, i.e., 
$$
  \pi_1(L_v[(u,s)])=\pi_1([(vu,s)])=vu\cdot\SO(3)=L_v(\pi_1[(u,s)]).
$$
According  to \eqref{eq: so(5)_basis}, we fix the basis
$\{e_8,e_9,e_{10}\}$ for $\so(3)\subset \so(5)$ such that $[e_i,e_j]=\varepsilon_{ijk}e_k$, where $\varepsilon_{ijk}$ is the sign of the permutation $(i\;j\;k)$. 
\begin{lemma}\label{lem:alphaInv}
There is a unique invariant $\SO(3)$--connection on $P_1=\V\times_{\phi_1} \SO(3)$ given by
\begin{equation}\label{eq: invariant_connection}
\alpha=e^8\otimes e_8+e^9\otimes e_9+e^{10}\otimes e_{10}.
\end{equation} 
\end{lemma} 
\begin{proof}
Consider $\alpha: \so(5)\rightarrow \so(3)$ in the basis \eqref{eq: so(5)_basis} given by $\alpha=a_ie^i\otimes e_8+b_je^j\otimes e_9+c_ke^k\otimes e_{10}$.  From Lemma \ref{rem:linearMapCondition} (i), we have that 
$ 
\alpha=\alpha_0+ e^8\otimes e_8+ e^9\otimes e_9+ e^{10}\otimes e_{10}, 
$ 
where $\alpha_0=a_ie^i\otimes e_8+b_je^j\otimes e_9+c_ke^k\otimes e_{10}$, with $i,j,k\in\{1,\dots,7\}$, we are going to conclude that $\alpha_0=0$. From Lemma \ref{rem:linearMapCondition} (ii), for any $\ig\in \SO(3)$   and $X\in \so(5)$, we have  $\ig\alpha(X)\ig^{-1}=\alpha(\ig X\ig^{-1})$. Taking $X=e_i $ for each $ i\in\{1,\dots,7\}$, we denote 
\begin{equation}\label{eq:LemAlphaInv1}
\alpha(\Ad(\ig)(e_i)) =\alpha(\ig e_i\ig^{-1}) 
\end{equation} 
and
\begin{equation}\label{eq:LemAlphaInv2}
\Ad(\ig)(\alpha(e_i))=a_i\ig e_8\ig^{-1}+b_i\ig e_9\ig^{-1}+c_i\ig e_{10}\ig^{-1}.    
\end{equation}
Now, consider
$\ig_k$ and $h_k$ as in \eqref{eq:h}, equations
\eqref{eq:LemAlphaInv1} and \eqref{eq:LemAlphaInv2} become
$$
\alpha(\Ad(\ig_0)(e_2))=\alpha(\ig_0 e_2\ig_0^{-1})=\alpha(e_3)=a_3e_8+b_3e_9+c_3e_{10}, 
$$ 
and 
$$
\Ad(\ig_0)(\alpha(e_2)) =-a_2 e_8+b_2e_{10}-c_2 e_{9}.
$$
These imply $-a_2=a_3, -c_2=b_3,  c_3=b_2$. Analogously, for
$h_1$, we have
$$
\alpha(\Ad(\ig_1)(e_2))=\alpha(e_3)=-a_3e_8-b_3e_9-c_3e_{10}
$$ 
and 
$$
\Ad(\ig_1)(\alpha(e_2)) =a_2 e_8+b_2e_{10}-c_2 e_{9}, 
$$
consequently, 
$-a_2=a_3, c_2=b_3, -b_2=c_3$, hence  $c_2=b_3=c_3=b_2=0$. Similarly: 
\begin{itemize}
\item From
$\Ad(\ig_0)(\alpha(e_4))=\alpha(\Ad(\ig_0)(e_4))$, we obtain $c_4=b_4=0$,
\item 
From  $\Ad(\ig_1)(\alpha(e_7))=\alpha(\Ad(\ig_1)(e_7))$, we obtain $c_7=b_7=0$. 
\item 
From $\Ad(\ig_0)(\alpha(e_5))=\alpha(\Ad(\ig_0)(e_5))$ and $\Ad(\ig_1)(\alpha(e_5))=\alpha(\Ad(\ig_1)(e_5))$, we obtain $c_6=b_5=c_5=b_6=0$.     
\end{itemize}
Finally, applying the same argument for 
$h_2$ and $h_3,$ we conclude that $a_i=0$ for $i=1,\cdots, 7$, consequently $\alpha_0=0$, then   $\alpha=e^8\otimes e_8+e^9\otimes e_9+e^{10}\otimes e_{10}$ and  \eqref{eq: invariant_connection} defines a unique $\SO(5)$-invariant connection $\alpha\in \Omega^1(P,\so(3))$.
\end{proof}
In the basis \eqref{eq: on-basis_so5}, the curvature of \eqref{eq: invariant_connection} is given by
\begin{equation}\label{eq: curvature_form}
    F_\alpha= \frac{1}{y_2} \left((X^{23}+X^{56})\otimes e_8+(X^{24}+X^{57})\otimes e_9+(X^{34}+X^{67})\otimes e_{10}\right).
\end{equation}     
Notice that $F_\alpha:\fm\times\fm\rightarrow \so(3)$ is a $\SO(3)$-equivariant map, i.e., $F_\alpha\circ \Ad(\ig)=\Ad(\ig)\circ F_\alpha$ for any $\ig\in \SO(3)$, thus \eqref{eq: curvature_form} corresponds with the $\SO(5)$-invariant curvature $2$-form of the induced by the $\SO(5)$-invariant connection $\alpha\in \Omega^1(P,\so(3))$ (cf. \cite[Proposition 1.4.6]{cap2009}).\\
\subsection{Homogeneous $\gt$ and self dual contact instantons (SDCI)}
For the $\SO(5)$-invariant contact form $\eta=X^{1}$ on $\V$ and the $\SO(5)$-invariant metric given by \eqref{eq:sasakianmetric}, we consider the equation
\begin{equation}\label{eq: invariant_contact_instanton_equation}
    \ast_g(F_\beta\wedge\eta\wedge \omega)=\pm F_\beta, \qwhereq \omega=\frac{1}{2}d\eta=(X^{25}+X^{36}+X^{47}).
\end{equation}

From Lemma \ref{lem:alphaInv}, we obtain the following 
\begin{corollary}
Consider $\V$ with the Sasakian structure given in Theorem \ref{thm:V52 Sasakian1}. Then, the connection \eqref{eq: invariant_connection} is an invariant  SDCI  on the homogeneous $\SO(3)$--principal bundle $P_1\to \V$.    
\end{corollary}
\begin{proof}
For any Sasakian manifold, the bundle of $2$--forms decomposes into \cite[Section~2.1]{portilla2023}
\begin{equation}\label{eq:split2forms}
\Omega^2=\Omega^2_1\oplus\Omega^2_6\oplus\Omega^2_8\oplus\Omega^2_V,    
\end{equation}
where \eqref{eq:sdci basis} are elements of $\Omega^2_8=span\{w_i\}_{i=1}^8$ solving $\ast(\eta\wedge \omega\wedge  w_i)=w_i$. Comparing \eqref{eq: curvature_form} with the values of $w_i$ in  \eqref{eq:sdci basis}, we note that  $F_\alpha$ in \eqref{eq: curvature_form} is given by
$$
F_\alpha:= \frac{1}{y_2}(w_1\otimes e_8+ w_3\otimes e_9+w_5\otimes e_{10}).
$$
The curvature tensor $F_\alpha$ satisfies $\ast(F_\alpha\wedge\eta\wedge \omega)=F_\alpha$.
Hence, for any $y_2\in\R^+$ satisfying \eqref{eq:SasakiConditionY_i}, the connection  \eqref{eq: invariant_connection} is an $\SO(5)$-invariant SDCI.
\end{proof}
\begin{corollary}\label{cor:P0}
Consider $\V$ endowed with the Sasakian structure given in Corollary \ref{cor:V52 Sasakian}.  There exist a unique invariant  ASDCI  on the trivial homogeneous $\SO(3)$--principal bundle $P_0\to \V$ (cf. Lemma \ref{lem:ASDCIonP0}).
\end{corollary}
\begin{proof}
From Lemma  \ref{lem:ASDCIonP0}, we only should note that the curvature \eqref{eq: SO(3)_curvature_trivial} satisfy $\ast_g(F_\beta\wedge\eta\wedge \omega)=F_\beta,$ if and only if $y_1=\frac92$.       
\end{proof}
A result analogous to  Corollary \ref{cor:P0} follows for $P\to \V$ in Remark \ref{rem:U(1)}, since the curvature tensor in \eqref{eq: U(1)_curvature} satisfy $\ast_g(F_\beta\wedge\eta\wedge \omega)=F_\beta,$ if and only if $y_1=\frac92$. 
\begin{proposition} 
Consider the connection \eqref{eq: invariant_connection} on  $\V$, with respect to  the $\gt$--structure given in Corollary \eqref{cor:V52 Sasakian}. Then, we have 
\begin{itemize}
\item 
The $\SO(3)$--connection in \eqref{eq: invariant_connection} on the homogeneous principal bundle $P_1\to \V$ [cf. Lemma \ref{lem:alphaInv}]   satisfies $F_\alpha\wedge\psi=0$ for any $y_1>0$, i.e., there exist a unique $\SO(5)$--invariant connection which is simultaneously $\gt$ and SDCI, given by \eqref{eq: curvature_form}. 
\item 
The $\SO(3)$--connection \eqref{eq:trivialConnection} on the trivial bundle $P_0\to \V$ (cf. Lemma \ref{lem:ASDCIonP0}), whose curvature is given in  \eqref{eq: SO(3)_curvature_trivial} satisfies $F_\alpha\wedge\psi=0$, for $\psi$ given in \eqref{eq: 4-form psi}, if and only if \eqref{eq:trivialConnection} is flat. 
\end{itemize}   
\end{proposition}

 \subsection{Yang--Mills connections}
Now, we analyse the Yang--Mills condition for instanton \eqref{eq: invariant_connection}, for a detailed study on stability of Yang--Mills connections we refer to \cite{bourguignon1981stability}. On the space of connections $\cA$,  the global inner product $(\alpha,\beta)=\int_M\langle \alpha\wedge \ast\beta\rangle$ allow us to define  the Yang--Mills functional as the $L^2$-norm of the curvature:

\begin{equation}\label{eq:ym}
\cym\colon \nabla\in\cA\to \Vert \rF_\nabla\Vert^2_{L^2}=\int_M\Tr(\rF_\nabla\wedge \ast \rF_\nabla)=\int_M\Vert \rF_\nabla\Vert^2 \in\R.
\end{equation}  
A critical point of \eqref{eq:ym} is called a \emph{Yang--Mills (YM)}  connection, i.e. $\nabla$ satisfies the so-called \emph{Yang--Mills equation} $d^\ast_\nabla F_\nabla=0$. 
Taking the exterior derivative $d_\nabla$ in \eqref{eq:introInstEquation} and using the Bianchi identity, one notes that a solution of the instanton equation \eqref{eq:introInstEquation}  satisfy the YM-equation with torsion:
$$  
d_\nabla(\ast F_\nabla)=d\sigma\wedge F_\nabla. 
$$  
Note that neither $\sigma=\eta\wedge d\eta$ nor $\sigma =\varphi$ are closed, consequently, the torsion term in the above equation is not automatically zero, however,  we have that SDCI and ASDCI are Yang--Mills connections \cite[Section 5]{luis2024Weitzenbock}.
With respect to the $\gt$-structure \eqref{eq: G2 Sasakian metric}, we have the following:

\begin{theorem}\label{thm: Yang-Mills example}
 Let $\alpha$ be the unique $\gt$-instanton  \eqref{eq: invariant_connection} on the homogeneous $\SO(3)$-principal bundle $P_1$. Then, $\alpha$ is a critical point for the Yang--Mills energy for any coclosed $\gt$-structure \eqref{eq: G2 Sasakian metric}.
\end{theorem}
\begin{proof}
 Using \eqref{eq: derivative varphi} and the fact that $\alpha$ is a $\gt$-instanton, we have that the curvature \eqref{eq: curvature_form} satisfies $d_\alpha^*F_\alpha=0$, if and only if
\begin{equation*}
2a^3-2a^2x-3ab^2+2aby+bxy=0, \qandq 2b^3-2b^2y-3a^2b+2abx+axy=0,
    \end{equation*}
    we see that both equations hold for \eqref{eq: G2 Sasakian metric}, i.e. $x=a$ and $y=b$. Hence, we obtain the result. 
\end{proof}
\section{Infinitesimal deformations and rigidity of $\gt$-instantons}\label{sec:spinors}
In this section, we extend the spinorial characterisation of infinitesimal
deformations of $\gt$--instantons obtained in \cite{singhal2022} to the setting
of coclosed $\gt$--structures. As an application, we show that for certain values of the parameter $y_2$ appearing in \eqref{eq: curvature_form}, the corresponding $\gt$--instanton is rigid. Throughout this section, $(M^7,\varphi)$ denotes a $7$--dimensional manifold
endowed with a coclosed $\gt$--structure, and $\zeta$ the unit spinor associated
with $\varphi$.
 
Let us recall, consider $\zeta$ the unit spinor associated to $\varphi$ and the induced spinor representation $\Delta_7\simeq \langle\zeta\rangle\oplus TM$, the Clifford product of a vector field $Y$ and a  spinor $(f,Z)$ is (see \cite{karigiannis2006some})
\begin{equation}\label{eq: Clifford_product}
  Y\cdot(f,Z)=(-\langle Y,Z\rangle,fY+Y\times Z),
\end{equation}
where $(f,0)=f\zeta$ and $Y\times Z=(\varphi(Y,Z,\cdot))^\sharp$. The Clifford product \eqref{eq: Clifford_product} is extended  to a $k$-form $\epsilon$ by 
\begin{equation}\label{eq: Clifford_product_forms}
\epsilon \cdot(f,Y)=\frac{1}{k!}\epsilon_{i_1,\dots i_k}e_{i_1}\cdot(e_{i_2}\cdot(\dots\cdot(e_{i_k}\cdot(f,Z)\dots))
\end{equation}
Hence, we have the characterisation of the instanton condition \eqref{eq:G2instantonIntro} in the $\gt$ case in terms of spinors:
 \begin{definition}\label{def:G2InstEspin}  
Let $P\to M$ be a principal bundle. A connection $A$ on a principal bundle $P$ is called a \emph{$\gt$--instanton} if its curvature satisfies $F_A \cdot \zeta = 0,$
where $\cdot$ denotes Clifford multiplication \eqref{eq: Clifford_product_forms}.
\end{definition}
\subsection{Infinitesimal  deformation of $\gt$-instantons}
\begin{definition}\label{def:DefInfinitesimal}
Let  $A$ be  a $\gt$-instanton on a principal bundle $P\to M$.
An \emph{infinitesimal deformation} of $A$ is a one--form
$\epsilon \in \Omega^1(M,\Ad(P))$ satisfying
\begin{equation}\label{eq:infinitesimal-deformation}
(d_A \epsilon)\cdot \zeta = 0,
\qquad
d_A^\ast \epsilon = 0.
\end{equation}
\end{definition}
Let $\nabla^g$ be the Levi--Civita connection of $g$
associated to $\varphi$, we consider the one-parameter family of connections
\begin{equation}\label{eq:conn_t}
\nabla^t = \nabla^g + \frac{t}{2}\, g^{-1} H_\varphi \qforq t\in\R,
\end{equation}
 and
\[
H_\varphi = \frac{\tau_0}{6}\,\varphi - \tau_3.
\]
For $t=1$, it turns out that  $\nabla^1$ is the unique connection with parallel skew-symmetric torsion $H_\varphi$ that makes $\varphi$ parallel \cite{friedrich2002}. The connection $\nabla^t$ induces a connection on the spinor bundle $\Spi$ such that:
\begin{equation}\label{eq: spinor_connection}
\nabla_X^t \mu=\nabla_X^g \mu + \frac{t}{4}(X \;\lrcorner\; H_\varphi)\cdot \mu,\qwhereq\mu \in \Gamma(\Spi)\qandq X \in \Gamma(TM)
\end{equation}
\begin{definition}\label{def:DiracOp1}
With respect to the  Clifford multiplication $\cdot\colon TM\times\Spi\to\Spi$, we define the Dirac operator
\begin{equation}\label{eq:DiracOp1}
    D^t \mu = \sum_{i=1}^7 e_i \cdot \nabla^t_{e_i} \mu,
\end{equation}
\end{definition}
Denote by $\nabla^{t,A}$ the product connection on $\Spi \otimes\Ad(P)$ 
induced by $\nabla^t$ and $d_A$. Hence, denote by $D^{t,A}$ the associated Dirac operator given by  $\nabla^{t,A}$, note that 
\begin{equation}\label{eq: Dirac_operator}
D^{t,A} = D^{0,A} + \frac{t}{4}\, g^{-1} \Hpi_\varphi,
\end{equation}
where
$g^{-1}\Hpi_\varphi(\mu)=\sum_{i=1}^7 e_i \cdot (e_i \;\lrcorner\; H_\varphi)\cdot \mu$. According to the irreducible decomposition of the $3$-forms $\Omega^3=\Omega^3_1\oplus\Omega^3_7\oplus\Omega^3_{27}$ (e.g \cite[\S 2.6]{bryant2003some}), there are the following identities:
\begin{lemma}\label{lem: Clifford identites}
Let $\zeta$ be the spinor associated with $\varphi$, then it holds
\begin{align*}
(X\lrcorner\varphi)\cdot \zeta=3X\cdot\zeta, \qquad (X\lrcorner\gamma)\cdot\zeta=2h(X)\cdot\zeta \qandq (X\lrcorner \ast\gamma)\cdot\zeta=-2h(X)\cdot\zeta,
\end{align*}
where $X\in\sX(M)$, $\gamma\in \Omega^3_{27}$ and $h$ is the traceless symmetric tensor defining $\gamma$ by 
$$
\gamma_{ijk}=h_{ip}\gamma_{pjk}+h_{jp}\gamma_{ipk}+h_{kp}\gamma_{ijp}.
$$
\end{lemma}
\begin{proof}
Now, using \eqref{eq: Clifford_product_forms}, we get
\begin{align*}
(e_i \;\lrcorner\; \varphi)\cdot \zeta 
&=\frac{1}{2}\,\varphi_{ijk}\, e^j \cdot e^k \cdot (1,0) 
= \frac{1}{2}\,\varphi_{ijk}\,(0,\, e^j \times e^k)
= \frac{1}{2}\,\varphi_{ijk}\,\varphi^{jk}{}_m\,(0,e^m)\\
&= 3g_{im}\,(0,e^m)= 3\,(0,e_i)=3e_i(0,1)\\
&=3e_i\cdot\zeta
\end{align*} 
    Here, we used the identity $\varphi_{ijk}\varphi_{mnk}=g_{im}g_{jn}-g_{in}g_{jm}+\psi_{ijmn}$ \cite{karigiannis2009flows}*{Lemma A.8}. Assume $\gamma \in \Lambda^3_{27}$, we compute
\begin{align*}
(X \;\lrcorner\; \gamma)\cdot \zeta &= X_i\,(e_i \;\lrcorner\; \gamma)\cdot \zeta =\frac12\, X_i\,\gamma_{ijk}\, e^j \cdot e^k \cdot \zeta\\
&= \frac12\, X_i\bigl(h_i{}^p\, \varphi_{pjk}+ h_j{}^p\, \varphi_{ipk}+ h_k{}^p\, \varphi_{ijp}\bigr)e^j \cdot e^k \cdot \zeta.\\
&=\frac12\, X_i\bigl(h_i{}^p\, \varphi_{pjk}+2 h_j{}^p\, \varphi_{ipk}\bigr)
e^j \cdot e^k \cdot (1,0)\\
&=\frac12\, X_i\bigl(h_i{}^p\, \varphi_{pjk}+ 2 h_j{}^p\, \varphi_{ipk}
\bigr)(0,\, e^j \times e^k)\\
&=\frac12\, X_i\bigl(h_i{}^p\, \varphi_{pjk}+ 2 h_j{}^p\, \varphi_{ipk}\bigr)(0,\, \varphi^{jk}{}_m e^m)\\
&=\frac12\, X_i\, h_i{}^p\,(6\,g_{pm})\,(0,e^m) + X_i\, h_j{}^p\bigl(g_{im} g^j{}_p- g_i{}^j g_{pm}+ \psi_{ipm}{}^j\bigr)(0,e^m)\\
&=3\, h(X)\cdot \zeta-h(X)\cdot \zeta.\\
&=2\, h(X)\cdot \zeta
\end{align*}
\end{proof}

The next result characterises the infinitesimal deformations of a $\gt$-instanton in terms of the operator \eqref{eq: Dirac_operator}.

\begin{proposition}\label{prop:infinitesimal-dirac}
Let $\epsilon\in \Omega^1(M,\Ad(P))$ and $D^{t,A}$ the Dirac operator \eqref{eq: Dirac_operator} with $A$ a $\gt$-instanton. Then $\epsilon$ defines an infinitesimal deformation of $A$ (cf. Definition \ref{def:DefInfinitesimal}), if and only if
\begin{align*}
D^{t,A}(\epsilon\cdot \zeta)=\frac{(-5+t)}{8}\,\tau_0\, \epsilon\cdot \zeta+ (3t-1)\,(\epsilon \;\lrcorner\; \tau_{27})\cdot \zeta.
\end{align*}
\end{proposition}

\begin{proof}
    Let $\{e_i\}_{i=1}^7$ be a local orthonormal frame on $M$.
By definition of the twisted Dirac operator associated with the product
connection $\nabla^{0,A}$, a direct computation gives

\begin{align*}
D^{0,A}(\epsilon\cdot \zeta) &= \sum e_i\cdot \nabla^{0,A}_{e_i}(\epsilon\cdot\zeta)
= (d_A \epsilon + d_A^\ast\epsilon)\cdot \zeta +\sum e_i\cdot \epsilon\cdot \nabla^{0,A}_{e_i}(\zeta)  \\
&= (d_A \epsilon + d_A^* \epsilon)\cdot \zeta - \sum_i\left(\frac{\tau_0}{24} e_i\cdot \epsilon\cdot ( e_i \lrcorner\varphi )\cdot\zeta-\frac14 e_i\cdot \epsilon\cdot ( e_i \lrcorner\tau_3 )\cdot\zeta\right). 
\end{align*}
Here, we used that $X\cdot Y+Y\cdot X=-2g(X,Y)1$ and $\nabla^1\zeta=0$ for the connection \eqref{eq: spinor_connection}. Now, using Lemma \ref{lem: Clifford identites} we get 
\begin{align*}
D^{0,A}(\epsilon\cdot \zeta) 
&= (d_A \epsilon + d_A^* \epsilon)\cdot \zeta - \frac{5\tau_0}{8}\,\epsilon\cdot \zeta-(\epsilon\lrcorner\tau_{27})\cdot\zeta.   
\end{align*} 
On the other hand, for
$H_\varphi = \frac{\tau_0}{6}\,\varphi - \tau_3$ we have
\begin{align*}
g^{-1}\Hpi(\epsilon\cdot \zeta)&=\sum_{i=1}^7 e_i \cdot (e_i \;\lrcorner\; H)\,(\epsilon\cdot \zeta) 
=\sum_{i,j,k,\ell}\frac{1}{2}\, H_{ijk}\, \epsilon_\ell\,
e^i \cdot e^j \cdot e^k \cdot e^\ell \cdot \zeta\\
&=\frac12\sum_{i,j,k}H_{ijk}\, \epsilon_\ell\cdot e^i \cdot e^\ell\cdot  e^j  \cdot e^k \cdot \zeta
+ \sum   H_{i\ell k}\epsilon_\ell e^i \cdot e^k\cdot \zeta -2\,(\epsilon \;\lrcorner\; H)\cdot \zeta\\
&= -\frac{1}{2}\sum H_{ijk}\epsilon_\ell\cdot e^\ell\cdot e^i \cdot e^j \cdot e^k \cdot \zeta - \sum H_{\ell jk}\, \epsilon_\ell  e^j \cdot e^k \cdot \zeta -4\,(\epsilon \;\lrcorner\; H)\cdot \zeta.\\
&= -3\, \epsilon \cdot H \cdot \zeta - 6\,(\epsilon \;\lrcorner\; H)\cdot \zeta
= -\frac{1}{2}\,\tau_0\, \epsilon \cdot \varphi \cdot \zeta - 6\,(\epsilon \;\lrcorner\; H)\cdot \zeta.\\
&= \frac{7}{2}\,\tau_0\, \epsilon \cdot \zeta
- \tau_0\,(\epsilon \;\lrcorner\; \varphi)\cdot \zeta
+ 6\,(\epsilon \;\lrcorner\; \tau_3)\cdot \zeta.\\ 
&=\frac12\, \tau_0\, \epsilon\cdot \zeta+ 12\,(\epsilon \;\lrcorner\; \tau_{27})\cdot \zeta.
\end{align*}
Consequently, from \eqref{eq: Dirac_operator} we obtain 
\[
D^{t,A}(\epsilon\cdot \zeta) =(d_A \epsilon + d_A^* \epsilon)\cdot \zeta -\frac{5}{8}\,\tau_0\, \epsilon\cdot \zeta - (\epsilon \;\lrcorner\; \tau_{27})\cdot \zeta+ \frac{t}{8}\,\tau_0\, \epsilon\cdot \zeta + 3t\,(\epsilon \;\lrcorner\; \tau_{27})\cdot \zeta.
\]
Assume that $\epsilon$ satisfies the \textit{linearised instanton equation} together with
the \textit{Coulomb gauge condition}, i.e., $d_A \epsilon = 0$ and $d_A^* \epsilon = 0$, then
\begin{equation}
D^{t,A}(\epsilon\cdot \zeta)=\frac{-5 + t}{8}\,\tau_0\, \epsilon\cdot \zeta
+ (3t-1)\,(\epsilon \;\lrcorner\; \tau_{27})\cdot \zeta.    
\end{equation}
\end{proof}
Since $\zeta$ is parallel with respect to $\nabla^1$, then
$D^{1,A}(\epsilon\cdot \zeta) = (D^{1,A} \epsilon)\cdot \zeta,$
in consequence, 
\begin{equation}
(D^{1,A} \epsilon)\cdot \zeta
=-\frac{\tau_0}{2}\, \epsilon\cdot \zeta
+ 2\,(\epsilon \;\lrcorner\; \tau_{27})\cdot \zeta.
\end{equation}
Hence, the infinitesimal deformation space of a $\gt$-instanton $A$
on a $7$-manifold with coclosed $\gt$-structure
is isomorphic to the kernel of the operator
\begin{equation}\label{eq:Opeerador}
 D^{1,A}
+ \frac{1}{2}\,\tau_0\, \mathrm{Id}_{\Omega^1(M,\mathrm{Ad}(P))}
- 2\, \tau_{27}^* \otimes \mathrm{Id}_{\mathrm{Ad}(P)}
\in \mathrm{End}\bigl(\Omega^1(M,\mathrm{Ad}(P))\bigr).   
\end{equation}
We have
$D^{t,A}= D^{1,A} + (t-1)\,\frac{1}{4}\, g^{-1} \Hpi.$
Consequently,
\[
D^{t,A}(\epsilon\cdot \zeta)=D^{1,A}(\epsilon\cdot \zeta)
+ (t-1)\,\frac{\tau_0}{8}\, \epsilon\cdot \zeta
+ (t-1)\,3\,(\epsilon \;\lrcorner\; \tau_{27})\cdot \zeta.
\]
Setting $t=\frac13$, we obtain $D^{1/3,A}(\epsilon\cdot \zeta) = D^{1,A}(\epsilon\cdot \zeta) - \frac{\tau_0}{12}\, \epsilon\cdot \zeta-2\,(\epsilon \;\lrcorner\; \tau_{27})\cdot \zeta$  and consequently, 
\begin{equation}\label{eq:key-operator-identity} 
D^{1/3,A} + \frac{7}{12}\,\tau_0\, \mathrm{Id}
= D^{1,A}+ \frac{1}{2}\,\tau_0\, \mathrm{Id}- 2\, \tau_{27}^* \otimes \mathrm{Id}_{\mathrm{Ad}(P)}
\end{equation}
\subsection{Weitzenböck inequality and curvature obstruction}
It is known the Weitzenböck–Lichnerowicz formula with torsion cf. \cite{Agricola2004Holonomy}
\begin{proposition}\label{prop:weitzenbock-torsion}
Let $EM$ be a vector bundle associated with a principal bundle $P$ and
let $\mu \in \Gamma(\Spi \otimes EM)$.
Let $A$ be any connection on $P$.
Then, for all $t \in \mathbb{R}$, the following identity holds:
\begin{equation}\label{eq:weitzenbock-torsion}
(D^{t/3,A})^2 \mu=(\nabla^{t,A})^\ast \nabla^{t,A}\mu
+ \frac{1}{4}\,\mathrm{Scal}\,\mu
+ \frac{t}{6}\, d\varphi \cdot \mu
- \frac{t^2}{18}\,\|\varphi\|^2 \mu
+ F \cdot \mu.
\end{equation}
\end{proposition}
As an application of Propositions~\ref{prop:infinitesimal-dirac} and~\ref{prop:weitzenbock-torsion}, we prove the following rigidity result.
\begin{corollary}\label{prop:rigidity}
Let $(V^{5,2},\varphi)$ be the Stiefel manifold endowed with the coclosed
$\gt$--structure given in Corollary \ref{cor:V52 Sasakian} and let $A$ denote the $\gt$--instanton \eqref{eq: invariant_connection} on the $\SO(3)$-bundle $P_1=\V\times_{\phi_1} \SO(3)$. Then, for $y_2>0$ as in  \eqref{eq: curvature_form} there exists a non--empty interval
$y_2 \in I$ such that,  for all values $y_2 \in I$, the $\gt$--instanton $A$ is rigid for these values of the metric parameter $y_2$.
\end{corollary}
\begin{proof}
We consider~\eqref{eq:weitzenbock-torsion} for $t=1$. Taking $\epsilon\cdot\zeta \in \Gamma(\Spi\otimes\Ad(P))$, we obtain
\begin{equation}\label{eq:norm13}
\begin{aligned}
\bigl|\nabla^{1/3,A}(\epsilon\cdot\zeta)\bigr|^2
&=
\left\langle (D^{\frac13,A})^2(\epsilon\cdot\zeta),\, \epsilon\cdot\zeta \right\rangle
- \frac{1}{4}\,\mathrm{Scal}\,|\epsilon\cdot\zeta|^2  
- \frac{1}{18}\,\langle d\varphi\cdot(\epsilon\cdot\zeta),\, \epsilon\cdot\zeta\rangle\\
&\quad
+ \frac{1}{162}\,\|\varphi\|^2 |\epsilon\cdot\zeta|^2
- \langle F_A(\epsilon\cdot\zeta),\, \epsilon\cdot\zeta\rangle .
\end{aligned}
\end{equation}
From Proposition~\ref{prop:infinitesimal-dirac}, if   $\alpha\in\Omega^1(M,\Ad(P))$ is an infinitesimal deformation of the
$\gt$--instanton $A$, then $\epsilon$ satisfies $D^{1/3,A}(\epsilon\cdot\zeta)=-\frac{7}{12}\,\tau_0\, \epsilon\cdot\zeta,$ hence  
\begin{equation}\label{eq:D13-square}
\left\langle (D^{\frac13,A})^2(\epsilon\cdot\zeta),\, \epsilon\cdot\zeta \right\rangle=\frac{49}{144}\,\tau_0^2\, \vert \epsilon\cdot\zeta\vert^2.
\end{equation}
From \eqref{eq: G2 Sasakian metric}, we compute 
\begin{align*}
d\varphi&=2\left(X^{2356} + X^{2457} + X^{3467}
\right)-\frac{3 P_1}{2y_2}
\left(X^{1237} - X^{1246} + X^{1345} - X^{1567}
\right)\\
&+\frac{3P_3 }{2y_2}
\left(X^{1234} - X^{1267} + X^{1357} - X^{1456}\right).
\end{align*}
where $P_1=a(a^2+b^2)^{-1/2}$ and $ P_3=b(a^2+b^2)^{-1/2}$. Notice that   $P_1^2 + P_3^2 = 1 $ and $P_1 + P_2 \le 1$. On the other hand, from the proof of Theorem \ref{thm:Nearly}, we have
$\tau_0= \frac{6}{7}
\left( 2 + \frac{1}{y_2}\right) $ and therefore,
\begin{align*}
\ast\tau_3 &= d\varphi - \tau_0 \psi \\
=&
\left( \frac{26}{7} + \frac{6}{7y_2}\right) \left( X^{2356} + X^{2456} + X^{3467} \right)-
\left(\frac{6}{7} + \frac{33}{14}\frac{P_1}{y_2}\right)
\left(X^{1237} - X^{1246} + X^{1345} - X^{1567}\right)\\
&+
\left(\frac{6}{7} + \frac{33}{14}\frac{P_3}{y_2}
\right)\left(X^{1234} - X^{1267} + X^{1357} - X^{1456}\right).
\end{align*}  
Hence $\lvert \tau_3 \rvert^2 =
3\left(\frac{26}{7} + \frac{6}{7y_2}\right)^2
+4\left(\frac{6}{7} + \frac{33}{14}\frac{P_1}{y_2}
\right)^2
+4\left(\frac{6}{7} + \frac{33}{14}\frac{P_3}{y_2}
\right)^2$ and using 
$
P_1 + P_3 \le P_1^2 + P_3^2 = 1, 
$ we get
\begin{align}\label{eq: norm tau3}
\lvert \tau_3 \rvert^2
\le\frac{4}{49}\left(241
+276\,\frac{1}{y_2}+\frac{1125}{4}\,\frac{1}{y_2^2}\right).
\end{align} 
From \eqref{eq:norm13},  using  \eqref{eq:D13-square} 
\begin{align*}
(*)&\le
\left(
\frac{49\lvert \tau_0 \rvert^2}{144}
- \frac{\mathrm{Scal}}{4}
- \frac{\lvert \tau_0 \rvert}{18}
+ \frac{7}{162}
+ \frac{1}{18}\lvert \tau_{3} \rvert
+ \lvert F_A \rvert
\right)
\lvert \epsilon  \cdot \zeta \rvert^2.
\end{align*} 
Thus, from \eqref{eq: curvature_form} and \eqref{eq: norm tau3}, we obtain
\begin{equation}\label{eq:weitzenbock_inequality}
\begin{aligned}
0 &\le\left(\frac{49\lvert \tau_0 \rvert^2}{144}- \frac{\mathrm{Scal}}{4}
- \frac{\lvert \tau_0 \rvert}{8}+ \frac{7}{162} 
+ \frac{\lvert \tau_3 \rvert}{18}+ \lvert F_A \rvert\right)
\lvert \epsilon  \rvert^2\\
&\le\Bigg[\frac{1}{4}\left(2 + \frac{1}{y_2}\right)^2+3-
\frac{6}{y_2}
-\frac{3}{28}\left( 2 + \frac{1}{y_2} \right)+\frac{7}{162}
+\frac{1}{63} \sqrt{241+\frac{276}{y_2}+\frac{1125}{4y_2^2}}
+\frac{\sqrt{3}}{y_2} \Bigg] \lvert \epsilon  \rvert^2.
\end{aligned}
\end{equation}
Setting $x=\frac{1}{y_2},$ the right–hand side of \eqref{eq:weitzenbock_inequality} can be written as $f(x)\,|\epsilon|^2,$ where for $x\ge0$
\[
\begin{aligned}
f(x):=&\frac14(2+x)^2+3-6x-\frac{3}{28}(2+x)+\frac{7}{162}
+\frac{1}{63}\sqrt{241+276x+\frac{1125}{4}x^2}+\sqrt3\,x 
\end{aligned}
\]
The function $f$ is continuous on $[0,\infty)$, moreover, a direct evaluation shows that $f(x)$ becomes negative on a bounded interval. In particular,
$f(x)<0 \quad \text{for all } x\in(\frac53,10).$\footnote{This interval is not maximal} Equivalently,
$f(y_2)<0$ for all  $y_2\in\left(\tfrac{1}{10},\tfrac{3}{5}\right)$. 
For such values of the metric parameter $y_2$, the right--hand side of the
Weitzenböck inequality becomes strictly negative (see Figure~\ref{fig:rigid}),
while the left--hand side is non--negative. This yields a contradiction unless
$\epsilon \equiv 0.$
Therefore, the instanton is rigid for this range of metric parameters.
\begin{figure}[h]
\centering
\begin{tikzpicture}
\begin{axis}[
    width=0.65\textwidth,
    height=0.45\textwidth,
    axis lines=middle,
    xmin=-5, xmax=15,
    ymin=-7, ymax=30,
    xlabel={$x$},
    ylabel={$f(x)$},
    ticks=none,
    samples=300,
    domain=-5:15,
]
\addplot[thick] 
{ 1/4*(2+x)^2 + 3 - 6*x - 3/28*(2+x) + 7/162
  + 1/63*sqrt(241 + 276*x + 1125/4*x^2)
  + sqrt(3)*x };
\end{axis}
\end{tikzpicture}\caption{Plot of the right--hand side of the Weitzenböck inequality as a function of the metric parameter $y_2$. This shows that the expression becomes  strictly negative for a non-empty interval of positive values of $y_2$.} \label{fig:rigid} 
\end{figure}

\end{proof}
 \appendix 
\bibliography{biblio}

\end{document}